\documentclass{siamltex}
\usepackage[latin1]{inputenc}
\usepackage[english]{babel}
\usepackage{graphicx}
\usepackage{amsmath,amsfonts,amssymb}
\usepackage{bbm,wasysym,stmaryrd}
\usepackage{framed}
\usepackage{subfigure}
\usepackage{color}
\usepackage{hyperref}
\usepackage[normalem]{ulem}

\newcommand{\R}{\mathbb{R}}
\newcommand{\N}{\mathbb{N}}

\newcommand{\der}[2]{ \frac{\text{d} #1}{\text{d} #2} }  %% Partial derivative
\newenvironment{remark} {\par {\noindent \it \sc Remark.} \small \it } {}

\newcommand{\eps}{\varepsilon}

\newtheorem{Assump}{Assumption}

\graphicspath{{Figures/}}

\title{Wild oscillations in a nonlinear neuron model with resets: (I) Bursting, spike adding and chaos}

\author{Jonathan E. Rubin\footnotemark[1] \and Justyna Signerska-Rynkowska\footnotemark[6] \footnotemark[3] \footnotemark[4] \and Jonathan D. Touboul\footnotemark[6] \footnotemark[4]
\and Alexandre Vidal\footnotemark[5] \footnotemark[4] }
\renewcommand{\thefootnote}{\fnsymbol{footnote}}
\footnotetext[1]{Department of Mathematics, University of Pittsburgh, Pittsburgh, USA, jonrubin@pitt.edu}
\footnotetext[6]{The Mathematical Neuroscience Team, CIRB-Collège de France (CNRS UMR 7241, INSERM U1050, UPMC ED 158, MEMOLIFE PSL*), Paris, France}
\footnotetext[3]{Faculty of Applied Physics and Mathematics, Gda\'{n}sk University of Technology, Poland, \\jsignerska@mif.pg.gda.pl}
\footnotetext[4]{Mycenae, Inria, Paris, France}
\footnotetext[5]{Laboratoire de Math\'ematiques et Mod\'elisation d'\'Evry (LaMME), CNRS UMR 8071, Universit\'e d'\'Evry-Val-d'Essonne, France}

\begin{document}
\maketitle

\renewcommand{\thefootnote}{\arabic{footnote}}
\begin{abstract}
In a series of two papers, we investigate the mechanisms by which complex oscillations are generated in a class of nonlinear dynamical systems with resets modeling the voltage and adaptation of neurons. This first paper presents mathematical analysis showing that the system can support bursts of any period as a function of model parameters. In continuous dynamical systems with resets, period-incrementing structures are complex to analyze. In the present context, we use the fact that bursting patterns correspond to periodic orbits  of the \emph{adaptation map} that governs the sequence of values of the adaptation variable at the resets. Using a slow-fast approach, we show that this map converges towards a piecewise linear discontinuous map whose orbits are exactly characterized. That map shows a period-incrementing structure with instantaneous transitions. We show that the period-incrementing structure persists for the full system with non-constant adaptation, but the transitions are more complex. We investigate the presence of chaos at the transitions.  
\end{abstract}
\begin{keywords}
 spiking dynamics, hybrid dynamical systems, complex oscillations,  nonlinear dynamics,  spike-adding, period-incrementing, unimodal maps, chaos, neuronal bursting
\end{keywords}

\begin{AMS} 
37C27, % Periodic orbits of vector fields and flows
37B10, % Symbolic Dynamics
37C10, % Vector fields, flows, ordinary differential equations
%37C25, %Fixed points, periodic points, fixed-point index theory
37G15, % Bifurcations of limit cycles and periodic orbits
37G35, % Attractors and their bifurcations
37N25, % Dynamical systems in biology
92C20 % Neural biology

\end{AMS}

\newpage
\section*{Introduction} 

Neurons are excitable cells that communicate with each other through stereotyped electrical impulses, called action potentials or spikes. Because of the almost invariable shape of spikes, it is widely believed that the neural code is contained in the times at which spikes are fired and   in the types of spike patterns fired. In addition to phasic responses (returning to resting potential after a few spikes  in response to an input) and tonic firing (spiking repeatedly), other ubiquitous patterns of neuronal activity include \emph{bursting}, in which several spikes are produced in rapid succession, followed by a quiescent phase lacking spikes, and \emph{mixed-mode oscillations}, in which subthreshold oscillations alternate with active periods of one or more spikes.
In this paper and the sequel, we provide a fine description of bursting patterns, spike-adding transitions, and mixed-mode oscillations in hybrid integrate-and-fire systems, a form of neuron model exhibiting a range of desirable features in a framework that is simple enough to allow for mathematical analysis.
In particular, in this paper, we provide a novel rigorous demonstration of the existence of a period-incrementing cascade that extends away from the limit of timescale separation.

Hybrid integrate-and-fire systems combine nonlinear subthreshold dynamics accounting for the excitable properties of nerve cells together with discrete resets of the voltage, abstracting the complex yet stereotyped mechanisms of spike emission and reset that  generally occur at a much faster timescale. These models follow a tradition introduced more than a century ago~\cite{lapicque:07,brunel2007lapicque} and have attracted important interest of the computational neuroscience community in the past decade or so. Indeed, by decoupling the input integration and cellular excitability from the processes related to the emission of the spikes, they provide a mathematically convenient description of the neuronal dynamics that yields a precise representation of spike timing, conserves the main excitability properties of neurons,  exhibits a wide repertoire of behaviors~\cite{izhikevich:03,izhikevich:07,brette-gerstner:05,touboul:08,jimenez2013}, can be precisely fitted to data~\cite{jolivet-kobayashi-etal:08}, and allows for efficient computational implementation, which, for example, led to the first simulation of a network on the order of the mammalian brain size~\cite{izhikevich-edelman:08}.

In this paper, we focus on bursting patterns and transitions between them involving period adding and chaos, while in the companion we move on to mixed-mode oscillations.  
We have two main motivations for studying bursting and associated transitions in hybrid integrate-and-fire systems.  The first is that, in light of their utility and popularity, it is essential to harness the mathematical simplicity of hybrid integrate-and-fire systems in order to elucidate their mathematical properties in general.  The second is that bursting is an important biological phenomenon in a variety of neural contexts, for which certain theoretical aspects have proved difficult to study analytically in smooth dynamical system models.  
As reviewed in~\cite{izhikevich2006bursting}, bursting is found both in brain areas ranging from the neocortex (e.g., in pyramidal neurons of layer 5) to the hippocampus, the thalamus and other subcortical areas.  Bursting is hypothesized to contribute to many brain functions and states, as reviewed and detailed in~\cite{izhikevich:00,izhikevich2003bursts,oswald2004}; for instance, it may enhance the reliability and flexibility of information transmission~\cite{kepecs-lisman:02,kepecs-lisman:03}, support synchronization~\cite{belykh2005synchronization}  in contexts including slow synchronized sleep rhythms~\cite{destexhe1998}, drive automated behaviors such as respiration and locomotion \cite{marder:00,marder2001,lindsey2012}, promote hormone or neuromodulator release~\cite{schultz1998,tabak2011}, and play a role in pathologies such as epilepsy~\cite{prince:78} and parkinsonism~\cite{rubin2012}.
Moreover, the number of spikes per burst may itself be functionally significant and has been linked to the phase or slope of input signals received by  bursting neurons \cite{kepecs-lisman:02,kepecs-lisman:03,samengo2010}.

Because of both its biological significance and its mathematical complexity, bursting has been the subject of significant attention from the theoretical community (e.g., \cite{burstbook}). A key approach for the study of these behaviors takes advantage of the vastly different timescales between spike emission and input integration to develop a slow-fast decomposition of the dynamics. This method has led to the identification of minimal models of smooth differential equations supporting various possible forms of repetitive bursting and the classification of the geometry of distinct types of canonical bursting scenarios in smooth dynamical systems~\cite{rinzel,izhikevich:00}. While the geometric structure of bursts is now relatively well understood in this context, it remains quite difficult to characterize transitions between different bursting patterns and to analytically determine the existence of bursts with a prescribed period or number of spikes per burst in such systems. Indeed, smooth dynamical systems displaying bursting require at least 3 dimensions and elucidating the features of bursting orbits relies on detailed slow-fast analysis.  In particular, changes in the number of spikes arising in each cycle of a periodic burst pattern occur through complex transitions such as the {\it spike-adding} mechanism, in which a modulation of the value of a parameter leads to the addition of one spike per burst cycle.  Evidence for the existence of transient chaotic behaviors has been reported when the fast dynamics has specific features (e.g., in the case of a slow-fast transition induced by a homoclinic bifurcation \cite{Terman99} and in the case of a double homoclinic bifurcation  \cite{Desroches13}). Yet the difficulty of fully characterizing the global return mechanisms induced by a nonlinear flow of 3 or more dimensions generally has precluded detailed quantitative elucidation of the underlying sequence of orbit bifurcations hidden in the spike-adding transition.

A fundamental work of Rinzel and Troy~\cite{rinzel-troy:83} initiated the rigorous mathematical study of bursting and spike incrementing in the Belousov-Zhabotinsky reaction by reducing the analysis to the study of a one-dimensional map approximated by a discontinuous, piecewise linear map. A similar idea was used by Levi~\cite{levi:90} to provide a geometric explanation of the period-adding phenomenon observed experimentally in periodically forced neon tubes. Levi initiated the use of the properties of circle map for this study, which proved to be a particularly fruitful technique. These two approaches have been adapted in applications related to bursting in neurons~\cite{pakdaman:95,juan:10,jia:12,avrutin} and in this vein, a variety of discrete \emph{map-based models} of bursting neurons have been developed that provide a versatile and easily implementable description of neuronal dynamics~\cite{rulkov2002modeling,rulkov2004oscillations,medvedev2005reduction,manica2010} but generally remain abstract from the biological viewpoint (see the review~\cite{ibarz2011map}).  

The present manuscript applies the map-based approach to a hybrid dynamical system and provides a detailed analysis of spike adding and transitions between burst patterns. Previous work reported that these neuron models can produce bursts of various periods~\cite{naud-macille-etal:08,touboul:08,touboul-brette:09,brette-gerstner:05,touboul-brette:08,jimenez2013}, and numerical evidence suggested the presence of an underlying period-incrementing structure~\cite{touboul-brette:09}, in which bursts of distinct periods arise, separated by chaotic intervals, as reset voltage is progressively varied. This manuscript undertakes the rigorous study of the spike adding  in the hybrid dynamical system framework. In particular, we establish here a deep relationship between the dynamics of nonlinear integrate-and-fire models and  unimodal maps of the interval, whose dynamics has been extensively studied. Our approach combines previous studies on piecewise continuous maps with notoriously rich dynamics together with a slow-fast analysis.

The paper is organized as follows. We start in Section~\ref{sec:Model} by describing the model and summarizing relevant properties of nonlinear integrate-and-fire models,  in  particular those related to the adaptation map governing the sequence of values of the adaptation variable across successive resets. In Section~\ref{sec:SlowFast}, we consider the limit where adaptation is very slow compared to the voltage, in which case we show that the map governing consecutive resets of the hybrid model converges to a piecewise continuous map and we explicitly demonstrate the presence of spike adding (as in~\cite{rinzel-troy:83} and subsequent works). In the limit of timescale separation, the transition from bursts with $k$ spikes to bursts with $k+1$ spikes is instantaneous and no chaos appears. In section~\ref{sec:PerAdd} we prove the persistence of the period-incrementing structure related to spike adding when the adaptation variable is slow but not constant along trajectories (as predicted by numerical evidence in a general context in~\cite{pring-budd:10}).
Finally, we investigate in Section~\ref{Chaos} the behavior of the system at the  transitions between bursts of different periods, analyzing the emergence of different forms of chaos and establishing the presence of period doubling events. 

\section{Background and main results}\label{sec:Model}

In this section, we review the properties of the class of neuron models studied in this manuscript and in the sequel and summarize the main results that we attain in this paper. 

\subsection{Nonlinear bidimensional integrate-and-fire neuron models}

The class of nonlinear bidimensional models used in the present manuscript describes the excitable properties of the membrane voltage of the nerve cell, $v$, coupled to an adaptation variable, $w$, according to the ordinary differential equation
\begin{equation}\label{eq:SubthreshDyn}
  \begin{cases}
    \der{v}{t} = F(v) - w + I \\
    \der{w}{t} = \eps ( bv - w)
  \end{cases}
\end{equation}
where $\eps>0 $ accounts for the ratio of timescales between the adaptation variable and the voltage dynamics;  $b>0$ represents the steady state ratio of adaptation to voltage, and can be seen as the coupling strength between these two variables; $I$ is a real parameter modeling the input current received by the neuron; and $F$ is a real function accounting for the intrinsic dynamical properties of the cell membrane, typically provided by leak currents together with spike initiation currents that provide excitability.  Specific models may  differ in the form of  $F$; in particular, $F$ was assumed to be quadratic in~\cite{izhikevich:03} and exponential in~\cite{brette-gerstner:05}, while in~\cite{touboul:08} a quartic model was proposed that had the capacity to support stable subthreshold oscillations (i.e., periodic orbits with no spike). 

Mathematically, it was shown that qualitative properties of these systems do not strongly depend on the precise choice of the nonlinearity $F$ as along as a few minimal properties are satisfied~\cite{touboul:08b, touboul-brette:09}, which we assume to be the case.
\begin{Assump}\label{Assump:Converge} The map $F:\mathbb{R}\to\mathbb{R}$ has the following properties:
\begin{itemize}
	\item it is regular (at least three times continuously differentiable);
	\item it is strictly convex;
	\item its derivative diverges at $+\infty$, i.e. $\lim\limits_{v\to\infty}F^{\prime}(v)=\infty$, and has a negative limit at $-\infty$ (possibly also  negative infinite) satisfying:
	\begin{equation}\label{EqforPlateau}
	\lim_{v\to -\infty}F^{\prime}(v)<-\eps (b+\sqrt{2});
	\end{equation}
	\item there exist $\eta, \alpha, \hat{v} > 0 $ such that $F(v)/v^{2+\eta}\geq \alpha$ for all $v\geq \hat{v}$.
\end{itemize} 
\end{Assump}

The first three hypotheses control the shape of the $v$-nullcline and hence constrain the possible  fixed points of the system. 
The last assumption ensures that, at least for some initial conditions, the membrane potential variable $v$ blows up in finite time, while the adaptation variable remains finite. Hence, the need to introduce an artificial arbitrary spike threshold is eliminated (contrarily to the case of the linear or quadratic adaptive models, for instance\footnote{In these cases, the voltage and adaptation variables blow up simultaneously~\cite{touboul:08c}. For such models, one needs to introduce a cutoff voltage defining the spike emission times. For such systems with finite voltage cutoffs, the results presented in this manuscript remain valid and should be simple consequences of the present work.}, see ~\cite{touboul:08c}. The blow up of $v$ is interpreted as the time of a spike. Following a voltage blow up at time $t_{*}$, the voltage is instantaneously reset to a constant value $v_R$ and the adaptation variable is updated as follows:
\begin{equation}\label{eq:ResetDyn}
  v(t) \xrightarrow[t\to t_{*}]{} \infty \implies \begin{cases}
                                            v(t_*) = v_R \\
                                            w(t_*) = \gamma w(t^{-}_{*})+ d
                                           \end{cases}
\end{equation}
with $\gamma\leq 1$ and $d\geq 0$, corresponding to the effect on the adaptation variable of the emission of a spike. 
With this reset mechanism, it is not hard to show that the system is globally well-posed, i.e. that one can define a unique forward solution for all times and initial conditions. This property requires that spikes do not accumulate in time, which can be demonstrated as done in~\cite{touboul-brette:09} for $\gamma=1$. 

Note that in all models in the literature it is assumed that $\gamma=1$. However, going back to the biological problem, spikes are not Dirac masses but stereotypical electrical impulses $ s(t)=\frac{1}{\delta t} S(\frac{t}{\delta t})$ where $S(t)$ is the typical spike shape rescaled on the dimensionless interval $[0,1]$, and $\delta t$ is the spike duration, generally small compared to the input integration timescale: $0<\delta t \ll 1/\varepsilon$. The adaptation variable integrates this sharp impulse:
\[w(t^*+\delta t) = w(t^{*-})e^{-\varepsilon \delta t} + \int_0^{\delta t } b s(u) e^{-\varepsilon (\delta t -u)}\,du = \gamma w(t^{*-}) + d\]
with $\gamma = e^{-\varepsilon \delta t} <1$ and $d=b\int_0^1 S(w)e^{-\varepsilon \delta t (1-w)}\,dw$. The classical nonlinear integrate-and-fire neuron of~\cite{brette-gerstner:05,izhikevich:04,touboul:08} corresponds to the limit $\delta t\to 0$ and hence $\gamma \to 1$.
In this paper, we take $\gamma=1$, while in the companion we do not. 

The excitability properties of the system governed by the subthreshold system~\eqref{eq:SubthreshDyn} were investigated exhaustively in~\cite{touboul:08}. It was found that all models undergo a saddle-node bifurcation and a Hopf bifurcation, organized around a Bogdanov-Takens bifurcation, along curves that can be expressed in closed form.
\begin{figure}[h]
	\centering
		\includegraphics[width=.8\textwidth]{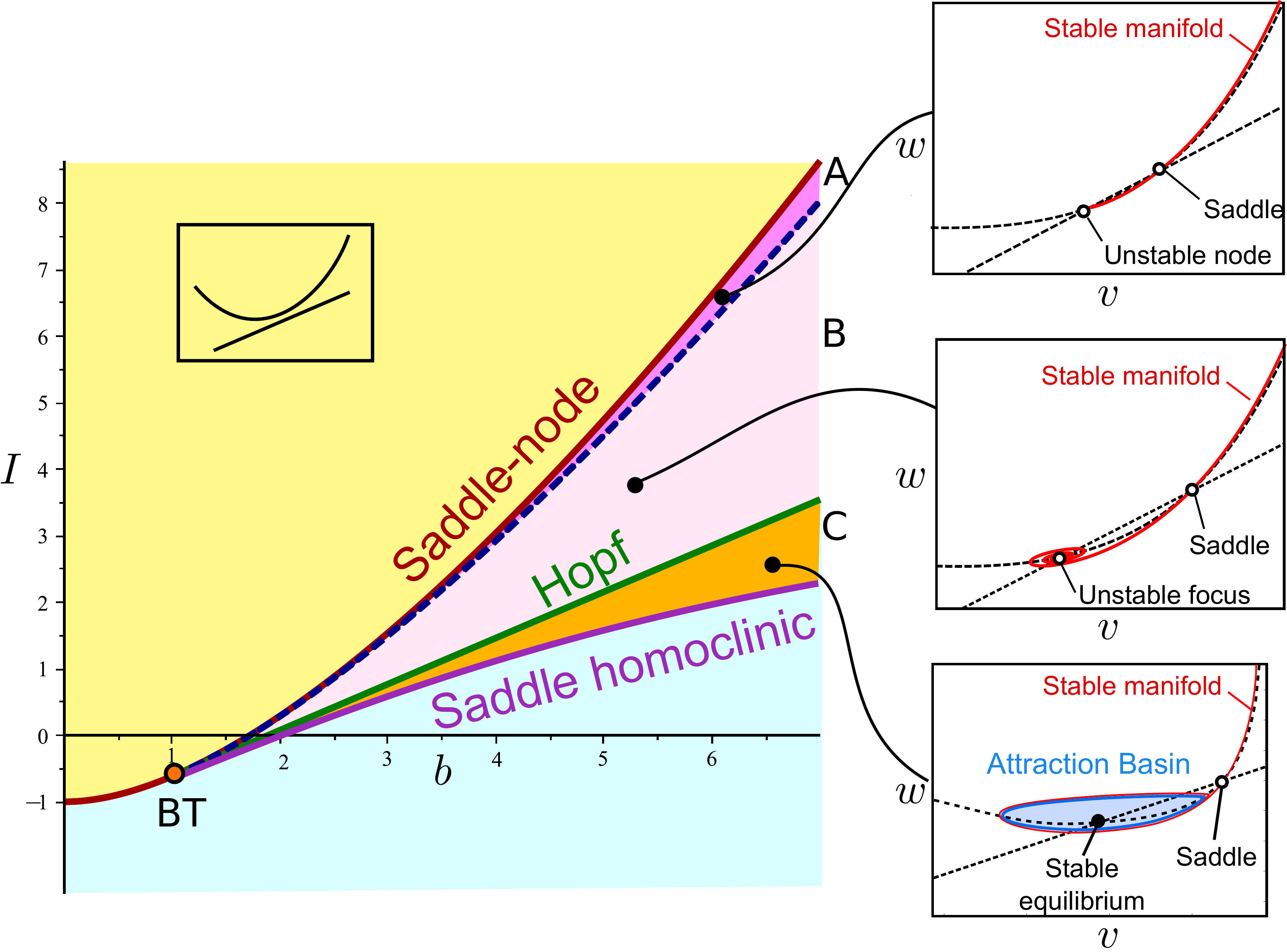}
	\caption{Bifurcations of the adaptive exponential model and its saddle-node (brown), Hopf (green), saddle homoclinic (purple) and Bogdanov-Takens (BT) bifurcations in the $(I,b)$ parameter plane. The analytical curve separating regions of unstable focus and unstable node is added in dashed blue. Typical phase planes in the different regions of interest are depicted as smaller insets. They feature the nullclines (dashed black) and the stable manifold (red).  }
	\label{fig:FPs}
\end{figure}
The different curves are depicted in the parameter space $(I,b)$ in Fig.~\ref{fig:FPs}. They split the parameter space into a region in which the system has no singular point (yellow region), a region in which the neuron has two singular points, one a stable steady state (orange and blue regions), and a region in which the system has two unstable singular points, one  a saddle and one repulsive (magenta and pink regions). In the latter case, the stable manifold of the saddle is made in part of a heteroclinic orbit connecting the saddle to the repulsive point. The heteroclinic orbit can either (i) wind around the unstable point  (pink region B), or (ii) monotonically connect to the repulsive point (magenta region A), depending on whether the eigenvalues of the repulsive point are real (A) or not (B). 
The companion paper will specifically focus on the dynamics of the system in case B.

This paper treats the case where the input is large enough so that the system has no equilibrium, represented by the yellow region in Fig.~\ref{fig:FPs}, which we ensure via the following assumption:
 \begin{Assump}\label{Assump:NoSingPt}
 The $w$-nullcline lies entirely below the $v$-nullcline, i.e.
 \begin{equation}
 \label{eq:noeq}
 \forall v, \quad F(v)+I>bv.
 \end{equation}
\end{Assump}
While numerics suggest that our results do not sensitively rely on this assumption, working under this condition simplifies the analysis, since the absence of equilibrium of the subthreshold dynamics implies that all initial conditions lead to spiking. 
In particular, this assumption allow us to define on the whole real line the \emph{adaptation map}~\cite{touboul-brette:09} that describes how the value of the adaptation variable evolves over consecutive resets on the reset line $\{v=v_{R}\}$. 
\begin{definition}\label{AdapMapDef}
The adaptation map $\Phi$ associates to any value $w\in \R$ the value of the adaptation variable after the spike and reset  of the trajectory starting at initial condition $(v_{R},w)$. Rigorously, if $(V(\cdot,v_R,w),W(\cdot,v_R,w))$ is the solution of equation~\eqref{eq:SubthreshDyn} with initial condition $(v_{R},w)$ and if $V$ blows up at $t_{*}$, then $\Phi(w):= W(t_{*},v_R,w)=\gamma W(t_{*}^{-},v_R,w)+d$.  In this paper, we take $\gamma=1$.
\end{definition}

This characterization obviously defines a unique value $\Phi(w)$ for any $w\in\R$, since under assumption~\ref{Assump:NoSingPt} every trajectory blows up in finite time and \sout{since} the value of the adaptation variable remains finite by assumption~\ref{Assump:Converge}, see~\cite{touboul:09}. A fine characterization of the shape and regularity of the adaptation map $\Phi$ has been provided and illustrates the central role of the value $w^{*}=F(v_{R})+I$ of the intersection of the reset line with the $v$-nullcline~\cite{touboul-brette:09}. Similarly, by $w^{**}=bv$ we will denote the intersection of the reset line with the $w$-nullcline. 
 We recall a few properties useful to the present analysis~\cite{touboul-brette:09}: 
\begin{description}
\item{P1.} $\Phi$ is increasing and concave on $(-\infty, w^*)$ (with $\Phi^{\prime\prime}(w)<0$ for $w<w^*$)
\item{P2.} $\Phi$ is decreasing and bounded below on $[w^*,\infty)$ and thus has an horizontal asymptote (plateau) at infinity, provided that  $\lim_{v\to -\infty}F^{\prime}(v)<-\eps(b+\sqrt{2})$.
\item{P3.} $\Phi$ is at least $C^{3}$ (more generally, $C^{k}$ if $F$ is as well)
\item{P4.} $\Phi$ has a unique fixed point in $\mathbb{R}$
\item{P5.} For all $w<w^{**}$, we have $\Phi(w) \geq w + d \geq w$.
\end{description}

The map $\Phi$ plays a central role in the analysis of the dynamics of the system since its iterations $\Phi^n(w)$ define the sequence of values of the adaptation variable after spikes, from which one can infer the spike pattern fired by the neuron and distinguish regular spiking, spike frequency adaptation, bursting or chaotic spiking~\cite{touboul-brette:09}. In the present case, the introduction of the adaptation map will reduce the analysis of the model to the study of a one-dimensional continuous unimodal map and thus allow the use of a number of well-developed tools from the theory of discrete dynamical systems, which we shall exploit in the present manuscript. 

Before we continue, it is important to note that the study of this map in the context of bidimensional integrate-and-fire models has seen recent developments. In particular, Foxall and collaborators established a relation between the adaptation map and transverse Lyapunov exponents to characterize the stability of spiking periodic orbits in nonlinear integrate-and-fire systems~\cite{foxall2012contraction}. In~\cite{jimenez2013}, a generalized linear integrate-and-fire system was investigated via a similar map that is locally contractive, either globally or in a piecewise manner; conditions for spiking and bursting dynamics were established and bifurcations underlying transitions between solution patterns were studied.
In the accompanying paper~\cite{paper2}, we address the question of the dynamics of the system in the presence of two unstable equilibria of the subthreshold dynamics; in that case, the map $\Phi$ is no longer continuous, and the rotation theory for discontinuous maps is applied to characterize the dynamics. The present study focuses on characterizing a sequence of period-incrementing bifurcations associated with spike-adding transitions in bursting solutions. The main results of the present study are summarized below. 

\subsection{Summary of the main results}
Numerical evidence suggests that, as $v_{R}$ increases, the adaptation map $\Phi$ undergoes a sequence of period-incrementing bifurcations characterized by the presence of chaotic transitions (see~\cite{touboul-brette:09}), as depicted in Fig.~\ref{fig:Bif} in the case of the quartic model $F(v)=v^{4}+2 a v$ with the standard parameter set (used throughout the manuscript except otherwise specified):
\begin{equation}\label{eq:StandardParams} 
F:v\mapsto v^{4}+2av,\quad a=0.2, \quad b=0.7,\quad I=2, \quad d=1, \quad \eps=0.4 
\end{equation}
We explore this structure in detail in the present manuscript.
\begin{figure}[!h]
	\centering
	\includegraphics[width=.5\textwidth]{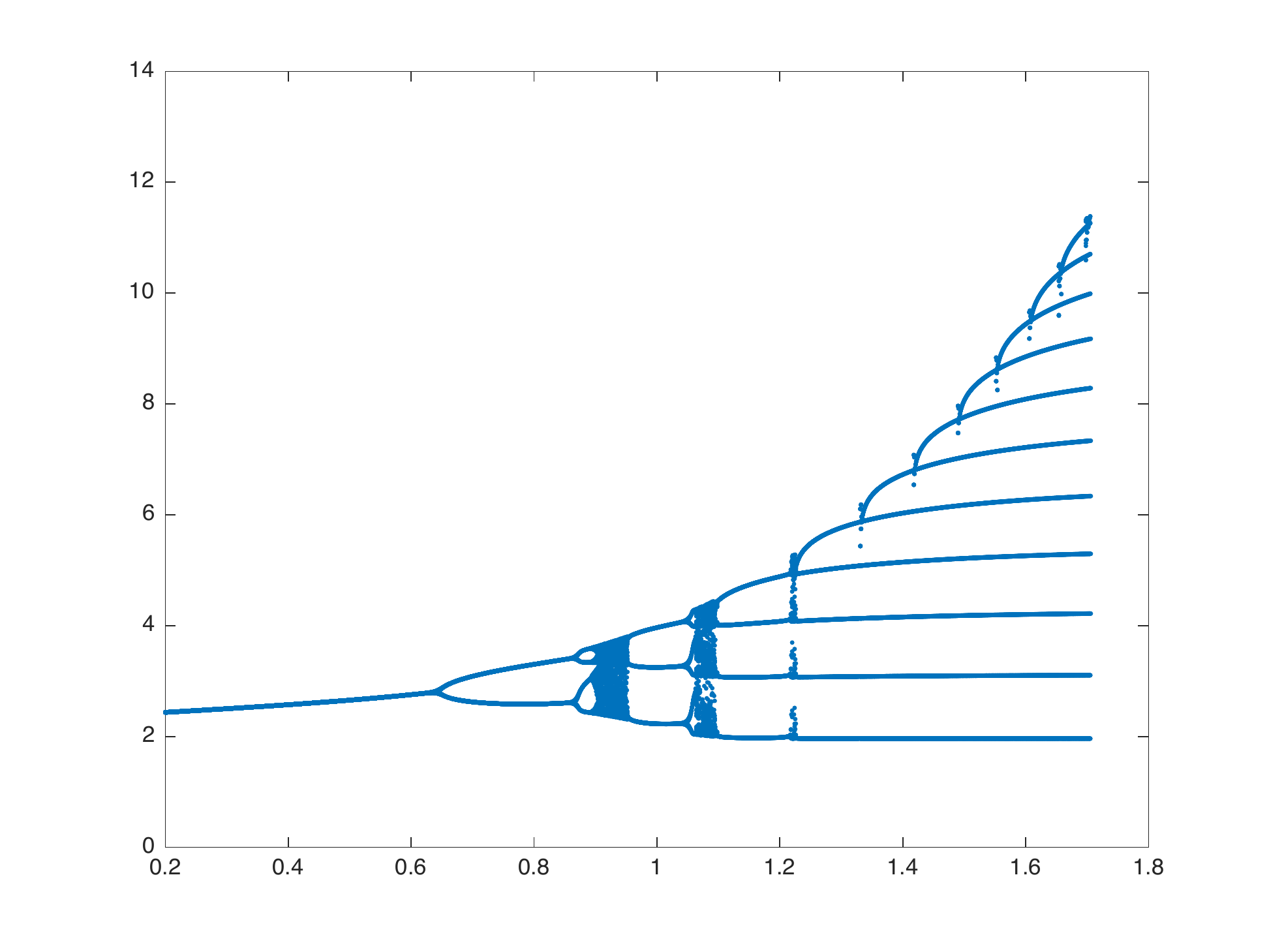}
	\caption{The period-incrementing structure of $\Phi$ as $v_{R}$ is varied for the quartic model (\ref{eq:SubthreshDyn} with parameters~\eqref{eq:StandardParams}). Blue points: 100 iterates of the map $\Phi$ after a transient of $100$ spikes (most points are overlapping and we only see the periodic points appearing, except in the chaotic region) as a function of $v_{R}$.}
	\label{fig:Bif}
\end{figure}

To establish the presence of a period-incrementing structure, we start by considering in Section~\ref{sec:SlowFast} the dynamics of the system in the limit of perfect timescale separation (i.e., $\varepsilon \to 0$). In that case, we will show that a number of properties of the map $\Phi$ can be obtained from the analysis of a simple piecewise linear map 
\begin{equation} \label{eq:phi0} \Phi_{0}:w\mapsto
\begin{cases}
	w + d & w\leq w^{*}\\
	p_{0} & w > w^{*} 
\end{cases}
\end{equation}
with 
\begin{displaymath}
p_0:=w_F+d,
\end{displaymath}
where $(v_F,w_F)$ are the coordinates of the unique minimum of the graph of $F+I$.  We assume that $w^*+d>p_0$. This map is particularly easy to analyze mathematically: it has a simple dynamics with a unique globally attractive periodic orbit with period $p= \left\lfloor\frac{w^{*}-p_{0}}{d}\right\rfloor+2$ (where $\lfloor \;  \rfloor$ denotes the integer part).
Since $w^{*}$ is an increasing function of $v_{R}$, the map will thus shows a period-incrementing bifurcation structure as a function of both parameters. 

We show that the adaptation map $\Phi$ converges towards $\Phi_{0}$ as $\eps\to 0$ in the Hausdorff distance (Proposition~\ref{limiting case} in Section~\ref{sec:SlowFast}) using standard slow-fast properties of the flow, and in the $C^{1}$ distance on any bounded interval not containing $w^{*}$ using more refined slow divergence estimates (Lemma~\ref{C1Lemma} in Section~\ref{sec:PerAdd}). This allows us to show in Section~\ref{sec:PerAdd} that the adaptation map $\Phi$ inherits the period-incrementing structure of $\Phi_{0}$ for small enough $\eps$. However, in contrast with the instantaneous bifurcations of $\Phi_{0}$, the period transitions are complex for $\eps>0$. The study of the dynamics in the vicinity of the transitions is performed in Section~\ref{Chaos}. We show that  $\Phi$ is topologically chaotic for a wide range of values of the reset voltage, for most reasonable definitions of chaos (e.g. chaos in the sense of Devaney or Block and Coppel, all of them being equivalent for the adaptation map). Of course, topological chaos is not necessarily reflected in the iterates of the map for non-specific initial conditions. This leads us in Section \ref{ChaosUnimod} to concentrate on the occurrence of the clearly visible chaotic bouts occurring at the transitions in the numerically obtained bifurcation diagrams. To this end, we characterize the presence of \emph{metric chaos}, corresponding to situations where the map features an invariant measure that is  absolutely continuous with respect to the Lebesgue measure.

\section{Period-incrementing in the singular limit}\label{sec:SlowFast}
We start by characterizing the bifurcations of the adaptation map in the singular limit of extreme separation of timescales $\eps\to 0$. 
In this section, we denote the map by $\Phi_{\eps}$ to emphasize its $\eps$-dependence.  $\Phi_{\eps}$ also depends on $v_R$, but we omit this dependence from our notation.
The limit of $\Phi_{\eps}$ as $\eps\to 0$ relies on a number of geometric invariants provided by Fenichel theory. Of particular importance, we will consider the {\em critical manifold} $\mathcal{C}$ formed by the singular points of the fast $v$-dynamics and given as a graph over $v$ by $w=F(v)+I$. Thanks to assumption \ref{Assump:Converge}, we know that $\mathcal {C}$ has  the unique fold point $(v_F,w_F) \in \mathcal {C}$ with $F'(v_F)=0$. This fold splits the critical manifold $\mathcal{C}$ into two branches, $\mathcal{C}^-$ defined for $v<v_F$ and $\mathcal{C}^+$ for $v>v_F$. 
For any point $(\bar{v},\bar{w})$ on $\mathcal{C}^-$ (resp. $\mathcal{C}^+$), $\bar{v}$ is an attractive (resp. repulsive) singular point for the fast dynamics $\dot{v}=F(v) - \bar{w} + I$ (considering $w$ as a parameter, $\bar{w}$). 

By Fenichel theory, we know that any compact submanifold $\mathcal{C}^- \cap \{v_1\leq v\leq v_2\}$, with $v_1<v_2<v_F$, perturbs, for small values of $\eps$, into a (non-unique) normally attractive invariant manifold for the flow of \eqref{eq:SubthreshDyn} lying in a $O(\eps)$-neighborhood of $\mathcal{C}^-$ (for the Hausdorff distance). As usual, the attractive slow manifold $\mathcal{C}^-_{\eps}$ is defined as the unique invariant manifold that is asymptotic to $\mathcal{C}^-$ when taking the limit $v \rightarrow -\infty$ and extended by the flow for positive time. Hence, any compact submanifold $\mathcal{C}^-_{\eps} \cap \{v_1 \leq v \leq v_2\}$ with $v_2<v_F$ is normally attractive for system \eqref{eq:SubthreshDyn} and $O(\eps)$-close to $\mathcal{C}^-$. Analogous results hold for $\mathcal{C}^+$, which defines the repulsive slow manifold $\mathcal{C}^+_{\eps}$ asymptotic to $\mathcal{C}^+$ for $v \rightarrow +\infty$ such that any compact submanifold $\mathcal{C}^+_{\eps} \cap \{v_3 \leq v \leq v_4\}$ with $v_F \leq v_3$ is normally repulsive for system \eqref{eq:SubthreshDyn} and $O(\eps)$-close to $\mathcal{C}^+$.  The slow manifolds, together with the points and notation that will arise in our analysis, are depicted in Figure~\ref{fig:PP-Map}.

Since we are interested in the bursting regime, we shall assume that:
\begin{Assump}\label{Assump3} The reset line is placed to the right of the fold $(v_F,w_F)$, i.e.
\begin{equation}\label{eq:vr_large}
	v_{R}>v_{F}.
\end{equation}
\end{Assump}

Let $p_{\varepsilon}$ denote the value of the plateau of $\Phi_{\varepsilon}$, i.e. $p_{\varepsilon}=\lim_{w\to\infty} \Phi_{\varepsilon}(w)$. For $\varepsilon$ small enough, for any $v_R$, trajectories emanating  from sufficiently large $w$ values along the reset line accumulate on the stable slow manifold $\mathcal{C_{\varepsilon}^-}$ and hence $p_{\varepsilon}$ is simply the value of the adaptation variable after reset for an initial condition on $\mathcal{C_{\varepsilon}^-}$.
Thus, beyond its linear dependence on $d$, the value of $p_{\eps}$ depends on $\eps$ but is independent of $v_R$. 

\begin{figure}[htbp]
	\centering
		\includegraphics[width=\textwidth]{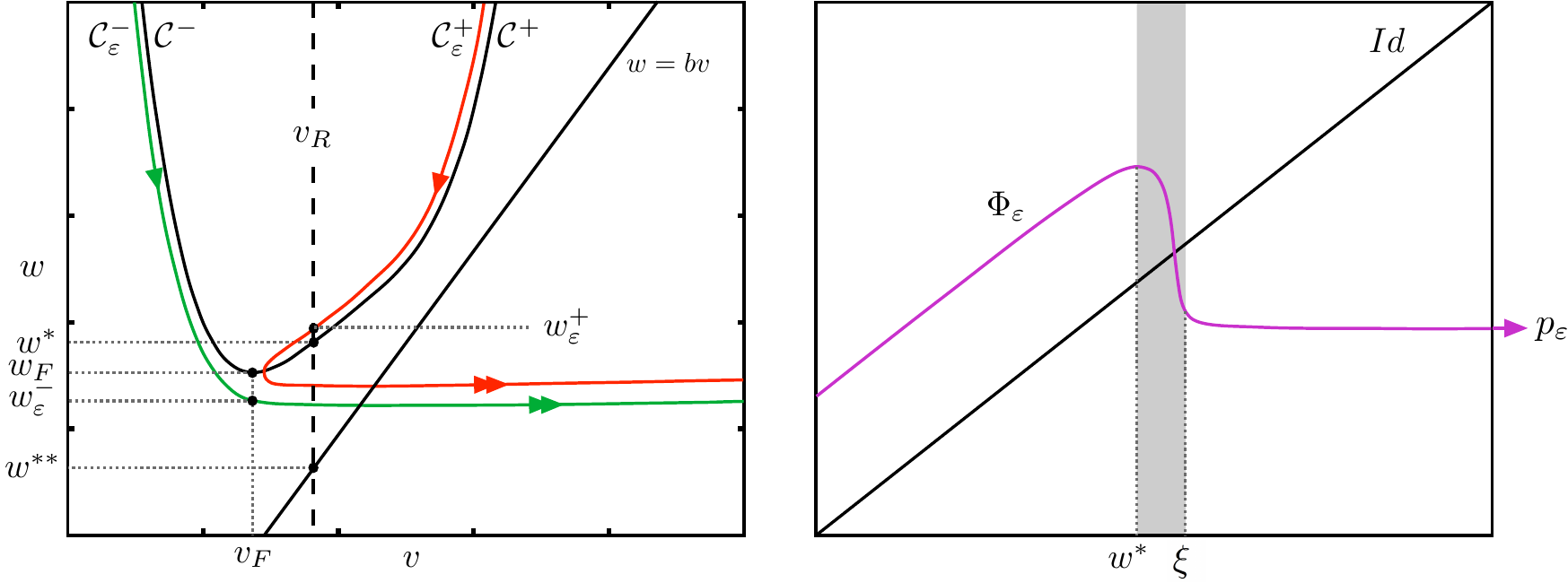}
	\caption{An example of the adaptation map for small $\eps$, annotated with notation used throughout the paper.  Left : attractive (green, $\mathcal{C_{\varepsilon}^-}$) and repulsive (red, $\mathcal{C_{\varepsilon}^+}$) slow manifolds of the subthreshold system as perturbations of the critical manifold $\mathcal{C}$ (black curve). As stated in the main text (in order from largest to smallest $w$):  $w_{\eps}^+$ denotes the $w$-coordinate of the upper intersection point of $\mathcal{C_{\varepsilon}^+}$ with $\{v = v_R\}$,  $w^*$ is the $w$-coordinate of $\mathcal{C} \cap \{ v=v_R \}$, $(v_F,w_F)$ denotes the minimum point of $\mathcal{C}$,  $w_{\eps}^-$ is the $w$-coordinate of  $\mathcal{C_{\varepsilon}^-} \cap \{ v = v_F \}$,  $w^{**}$ denotes the $w$-coordinate of the intersection of the $w$-nullcline (black line) with $\{ v = v_R \}$.   Right: Associated adaption map $\Phi_{\varepsilon}$. The grey region corresponds to the interval $[w^*,\xi]$, where $\xi$ is defined as the largest value such that $\Phi'(\xi) = -1$ (see section~\ref{sec:PerAdd}), while $p_{\varepsilon} := \lim_{w\to\infty} \Phi_{\varepsilon}(w)$.}
	\label{fig:PP-Map}
\end{figure}

We start by showing that $\Phi_{{\eps}}$ converges, in a suitable sense, towards the piecewise linear map $\Phi_0$ given in (\ref{eq:phi0}), with $p_0 = w_F+d$.
Indeed, for $\eps$ small enough, one notices heuristically that:
\begin{itemize}
\item for any $w\leq w^{*}$, the value of $w$ remains almost constant during the whole trajectory while $v$ blows up, implying that the value of $w$ at reset approaches $w+d$ as $\eps\to 0$; 
\item for any $w>w^{*}$, the trajectory approaches the stable slow manifold (which is very close to the $v$-nullcline) and follows it until it fires, and thus the adaptation variable approaches $p_{\eps}$, which in turn converges toward  $p_0$ when $\eps\to 0$. 
\end{itemize}

To rigorously make sense of the convergence of the family of continuous maps $\Phi_{\eps}$ towards the discontinuous piecewise linear map $\Phi_{0}$, we will rely on the notion of Hausdorff distance between the graphs of  functions. In detail, let us denote by $G(\Phi_{\varepsilon})$ the graph of the map $\Phi_{\eps}$, i.e. $\{ (w,y)\in \mathbb{R}^2:  \ y=\Phi_{\varepsilon}(w) \}$, and $G(\Phi_0)$ the graph of the function $\Phi_0$ augmented with the vertical segment at the discontinuity point $w_0$:
\[
G(\Phi_0):=\{(w,y)\in \mathbb{R}^2:  \ w\neq w^*, y=\Phi_0(w) \}\cup \{(w^*,y) \in \mathbb{R}^2: \  p_0\leq y\leq w^*+d\}.
\]
The {\em Hausdorff distance} between the sets $X:=G(\Phi_0)$ and $Z:=G(\Phi_{\varepsilon})$ is defined by:
\[
d_{\mathrm{H}}(X,Z)=\max\{ \sup_{x\in X} \inf_{z\in Z} d(x,z), \; \sup_{z\in Z} \inf_{x\in X} d(x,z)\,\},
\]
where $d(x,z)$ denotes the Euclidean distance between the points $x, z\in \mathbb{R}^2$. We show the following:

\begin{proposition}\label{limiting case}
For any fixed $v_{R}>v_{F}$, we have $d_{\mathrm{H}}(G(\Phi_{\varepsilon}), G(\Phi_{0})) \to 0$ as $\varepsilon\to 0$. Moreover, this convergence is uniform in $v_{R}$ on any compact interval $[v_{R_{1}},v_{R_{2}}]$ of $v_{R}$ values with $v_{R_1}>v_F$.
\end{proposition}

We split the proof into three steps: (i) in Lemma~\ref{limiting caseLemma}, we show the $C^0$-convergence of $\Phi_{\eps}$ towards $\Phi_{0}$ for $w<w^{*}$, (ii) in Lemma~\ref{BetaLemma}, we establish $C^0$-convergence on $w>w^{*}+\delta$ for some $\delta>0$, and (iii) we complete the proof by showing that as $\eps \to 0$, the value of $\delta$ can be chosen arbitrarily small.  For the proof, we will use the notation $w_{v_R}^*$ rather than just $w^*$ to emphasize the fact that this point depends on $v_R$.

Note that we will strengthen this result in Lemma~\ref{C1Lemma} by showing that this convergence is also true in a $C^{1}$ sense, i.e. that away from $w^{*}$, the differential of the adaptation map converges to the differential of $\Phi_{0}$ as $\eps\to 0$.

\begin{lemma}\label{limiting caseLemma} Given $[v_{R_1},v_{R_2}]$ as above, for every $\delta>0$, there exists $\tilde{\varepsilon}>0$ such that for all $\varepsilon<\tilde{\varepsilon}$, for every $v_R\in [v_{R_1},v_{R_2}]$ and for all $w\in (-\infty,w^*_{v_R}]$ we have $\vert \Phi_{\varepsilon}(w)-\Phi_0(w)\vert <\delta$.
\end{lemma}

\begin{proof} The proof proceeds by following the trajectories of \eqref{eq:SubthreshDyn} in the phase plane. For $w<w^*$ the solution starting from $(v_{R},w)$ remains below the $v$-nullcline, thus $v$ is stricly increasing along the whole trajectory. It is thus easy to show that we can express the trajectory of \eqref{eq:SubthreshDyn} as $W_{\eps}=W_{\eps}(v;v_R,w)$ for $v\geq v_R$, and where $W_{\eps}$ is the solution of the following differential equation:
\begin{equation}\label{eq:SubthreshDyn2}
  \begin{cases}
    \der{W}{v} = \frac{\eps ( bv - W)}{F(v) - W + I }\\
    W(v_R)= w.
  \end{cases}
\end{equation}
The distance we aim at evaluating is thus given by:
\[\vert\Phi_{\varepsilon}(w)-\Phi_0(w)\vert  =\vert \lim_{v\to \infty}W_{\varepsilon}(v;v_R,w) - w\vert=\varepsilon\vert\int_{v_R}^{\infty}\frac{bu-W_{\varepsilon}(u;v_R,w)}{F(u)-W_{\varepsilon}(u;v_R,w)+I}\;du\vert.\]
For $w< b v_{R}=:w^{**}$, we have $w\leq W_{\varepsilon}(u;v_R,w)<bv$ for all $v\geq v_{R}$, thus
\[\vert\Phi_{\varepsilon}(w)-\Phi_0(w)\vert \leq \varepsilon \int_{v_R}^{\infty}\frac{bu-w}{F(u)-bu+I}\;du \leq \eps\int_{v_{R_1}}^{\infty}\frac{bu}{F(u)-bu+I}\;du,\]
which is finite because of the integrability of $u/F(u)$ at infinity (Assumption~\ref{Assump:Converge}). The righthand side thus provides the uniform bound announced for $v_{R}\in[v_{R_1},v_{R_2}]$ and $w< bv_{R}$. 

For $w\in [w^{**},w^{*}]$, we decompose each trajectory into the segment $v\in [v_{R},\breve{v}]$, with $\breve{v}$ the value at which the trajectory $W_{\eps}$ crosses the $w$-nullcline, and the part $v\in [\breve{v},\infty)$: 
\begin{equation}
\vert\Phi_{\varepsilon}(w)-\Phi_0(w)\vert  
\leq  \int_{v_R}^{\breve{v}} \frac{\varepsilon(W_{\varepsilon}(u;v_R,w)-bu)}{F(u)-W_{\varepsilon}(u;v_R,w)+I}\;du + \vert \int_{\breve{v}}^{\infty} \frac{\varepsilon(bu-W_{\varepsilon}(u;v_R,w))}{F(u)-W_{\varepsilon}(u;v_R,w)+I}\;du\vert.
\end{equation}
The second term is handled similarly as the case $w<w^{**}$. As for the first term, using the fact that $v_{R}>v_{F}$ and  $v\in[v_{R},\breve{v}]$, we readily obtain: 
\[\frac{W_{\varepsilon}(u;v_R,w)-bu}{F(u)-W_{\varepsilon}(u;v_R,w)+I} \leq \frac{w-bv_{R}}{F(v_{R})-w+I}\]
and thus conclude:
\begin{multline}
\vert\Phi_{\varepsilon}(w)-\Phi_0(w)\vert \leq \eps \left(\frac{w-bv_{R_{1}}}{F(v_{R_{1}})-w+I} (\breve{v}-v_{R_{1}}) +\int_{v_{R_{1}}}^{{\infty}}\frac{bu}{F(u)-bu+I}\,du \right).
\end{multline}
In both cases, we find an upper bound proportional to $\eps$, and these bounds are continuous with respect to $v_{R}$.  Hence, given $\delta>0$, we can find $\eps$ sufficiently small to obtain that for every $v_{R}\in[v_{R_{1}},v_{R_{2}}]$, $\vert\Phi_{\varepsilon}(w)-\Phi_0(w)\vert<\delta$. \end{proof}

\begin{lemma}\label{BetaLemma} 
Given $[v_{R_{1}},v_{R_{2}}]$ with $v_{R_1}>v_F$, for any $\delta>0$, there exists $\tilde{\eps}>0$ such that for any $\eps<\tilde{\eps}$, for every $v_{R}\in [v_{R_{1}},v_{R_{2}}]$ and for all $w\in [w^{*}_{v_R}+\frac{\delta}{2},\infty)$, we have $\vert\Phi_{\varepsilon}(w)-\Phi_0(w)\vert<\frac{\delta}{2}.$
Moreover, $p_{\varepsilon}\to p_{0}$ as $\eps\to 0$. 
\end{lemma}
\begin{proof}  
We start by proving the convergence of $p_{\eps}$ using the expression
\[
p_{\varepsilon}=w^{-}_{\eps}+\lim_{v\to\infty}\int_{v_F}^{v}\frac{\varepsilon(bu-W(u;v_F,w^{-}_{\eps}))\;du}{F(u)-W(u;v_F,w^{-}_{\eps})+I}+d,
\]
where $w^{-}_{\eps}:=\mathcal{C}_{\varepsilon}^-(v_F)$ is the $w$-coordinate of the intersection of the invariant manifold $\mathcal{C}_{\varepsilon}^-$ with the line $\{ v=v_F \}$.  
Since the family of manifolds $(\mathcal{C_{\varepsilon}^-} \cap \{v\leq v_F\})_{\varepsilon}$ converges to the graph $\{w = F(v)+I,\; v\leq v_F\}$ as $\eps \to 0$, 
it follows that for any $\delta>0$, there exists  $\tilde{\varepsilon}$ small enough so that for $\eps\leq \tilde{\varepsilon}$
\[
0<w_F-w_{\eps}^{-}=F(v_F)+I - \mathcal{C}_{\varepsilon}^-(v_F)<\delta/4.
\]
It thus suffices to establish that the integral term vanishes as $\eps\to 0$, i.e. that for $\eps$ small enough we have
\[
\lim_{v\to\infty}\int_{v_F}^{v}\frac{\varepsilon(bu-W(u;v_F,w^{-}_{\eps}))\;du}{F(u)-W(u;v_F,w^{-}_{\eps})+I}<\delta/4,
\]
which indeed follows as a simple consequence of the integrability at infinity of $u/F(u)$ as in the proof of Lemma~\ref{limiting caseLemma}. 

With this result in hand we now prove the convergence of $\Phi_{\eps}$. Note that the backwards trajectory from $(v_F,w_F)$ intersects $\{ v=v_R \}$ in a point $(v_R,\tilde{w}_F)$ that converges to $\mathcal{C}$ as $\eps \to 0$.  For  $v_R>v_F$  and $\delta>0$, there exists $\tilde{\varepsilon}>0$ sufficiently small so that for any $\eps<\tilde{\eps}$, $w^*+\delta/2>\tilde{w}_F$.  It follows that
for any $\varepsilon\leq\tilde{\varepsilon}$  and $w\geq w^*+\delta/2$, the trajectory of $(V_{\varepsilon}(t;v_R,w),W_{\varepsilon}(t;v_R,w))$ with initial condition $(v_R,w)$ intersects the $v$-nullcline at some time $t_{\varepsilon}>0$ with $v$-coordinate $V_{\varepsilon}(t_{\varepsilon};v_R,w)<v_F$. For $t>t_{\varepsilon}$, $V_{\varepsilon}$ is strictly increasing and the trajectory remains bounded between the $v$-nullcline and  the forward trajectory from $(v_F,w_F)$. In particular, at $v=v_F$, we have ${w}_{\eps}^{-} < W_{\varepsilon}(t;v_R,w)<w_F$.
The same arguments as used above now apply to show that $\Phi_{\varepsilon}(w)$ is within a $\delta/2$-neigborhood of $\Phi_{0}(w)=p_0$.
 Again, as the manifold $\mathcal{C}_{\varepsilon}^{-}$ and the $v$-nullcline do not depend on $v_R$, this estimate holds uniformly in $v_R$ for $v_R$ varying in a compact interval.
 \end{proof}

\medskip

\noindent\emph{[Proof of Proposition~\ref{limiting case}]} Lemmas~\ref{limiting caseLemma}  and ~\ref{BetaLemma} taken together prove uniform convergence (i.e.,  in $C^0$-topology) of $\Phi_{\eps}$ towards $\Phi_{0}$ on $\R\setminus (w^{*}_{v_{R}},w^{*}_{v_{R}}+\frac \delta 2)$; therefore, these parts of the graphs $G(\Phi_0)$ and $G(\Phi_{\varepsilon})$ converge also in the Hausdorff metric. Moreover, the Hausdorff distance between the two maps for $w \in (w^*_{v_R}, w^*_{v_R}+\delta/2)$ is also bounded by $\delta/2$ since any point $(w,\Phi_{\varepsilon}(w))$ on the graph $G(\Phi_{\varepsilon})$ with $w^*_{v_R}<w\leq w^*_{v_R}+\delta/2$ lies within the rectangle $[w^*_{v_R},w^*_{v_R}+\delta/2]\times [p_0-\delta/2,w^*_{v_R}+d+\delta/2]$ and hence its distance to the graph $G(\Phi_0)$ does not exceed $\delta$. Altogether, we thus have:
\[
\forall \delta>0 \ \exists \tilde{\varepsilon}>0 \; \mbox{such that} \;  \forall {\varepsilon\leq \tilde{\varepsilon}} \; \mbox{and} \;  \forall v_R\in [v_{R_1},v_{R_2}] ,\quad d_{\mathrm{H}}(G(\Phi_0),G(\Phi_{\varepsilon}))<\delta.
\] 
$\square$

\begin{figure}[htbp]
	\centering
		\includegraphics[width=\textwidth]{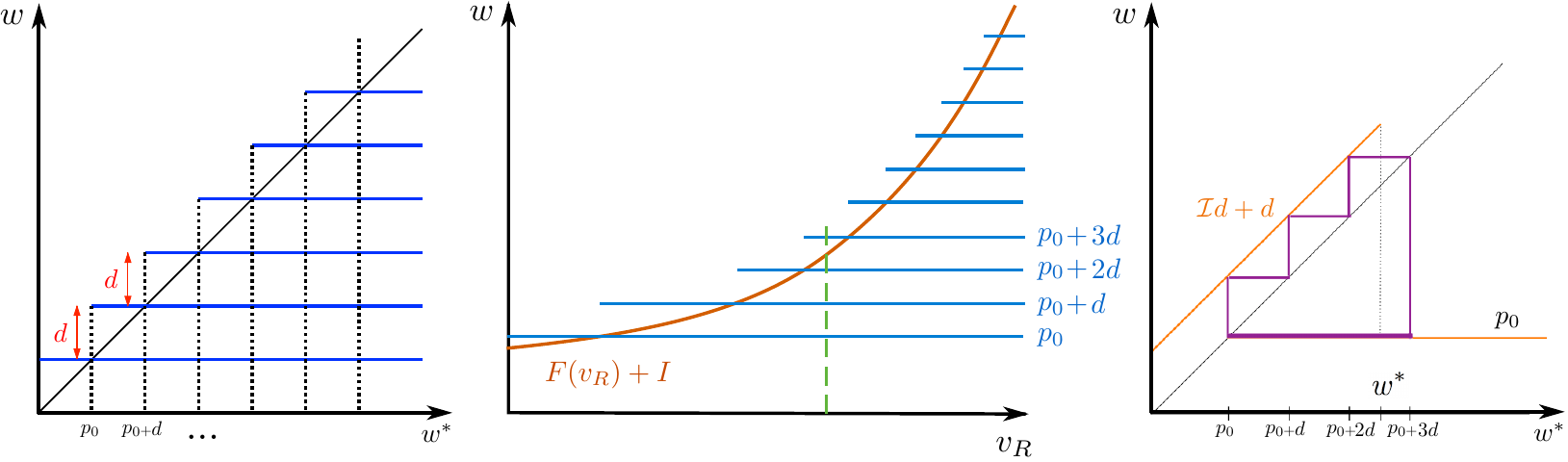}
	\caption{Bifurcation structure and periodic solutions of $\Phi_0$.  Left: bifurcation diagram of $\Phi_{0}$ as a function of $w^{*}$ illustrating the period-incrementing structure as described in Proposition~\ref{DynamicsLinear}. Blue lines represent the  stable periodic orbits of the system. Middle: bifurcation diagram of $\Phi_{0}$ as a function of $v_{R}$ in the case of the quartic model with standard parameters~\eqref{eq:StandardParams}. The red curve represents the value of $w^{*}$ for each $v_R$. Right: plot of $\Phi_{0}$ and the associated period-4 orbit for $v_{R}=1.3$ (green line in the middle plot).}
	\label{fig:graphPhi0AndBifLinear}
\end{figure}

We emphasize that the limit (in Hausdorff distance) of $\Phi_{\eps}$ at $\eps=0$ does not correspond to a reset map for $\eps=0$ since  the adaptation variable for the bidimensional system (\ref{eq:SubthreshDyn}) with $\varepsilon=0$ is not defined for $w >w^*$. Indeed,  the trajectories with initial condition $(v_R,w)$ with $w>w^*$ will simply converge towards some point $(\bar{v},w)$ on the critical attractive manifold $\mathcal{C}^{-}$ 
and therefore will not blow up and fire a spike. Thus in the statement of Proposition \ref{limiting case}, $\Phi_0$ is not the adaptation map for $\varepsilon=0$, but just the limit of maps $\Phi_{\varepsilon}$ for $\eps\to 0$.

Now that the limit $\Phi_0$ of $\Phi_{\eps}$ has been derived, we investigate the dynamics of $\Phi_0$ and establish that this map exhibits a period-incrementing phenomenon. 

\begin{proposition}\label{DynamicsLinear}
For any $v_{R}$, the map $\Phi_0$ has a unique periodic orbit, which is globally attractive and has period $p$ given by:
\begin{displaymath}
p:=\min\{k\in\mathbb{N}: \ p_0+(k-1)d>w^* \}. 
\end{displaymath}
With the increase of $v_R$ and hence $w^*$, the period of this orbit  is incremented by $1$ at each point  $w^*=p_0+(k-1)d$, $k\in\N$. The map thus displays a period-incrementing structure with instantaneous transitions. 
\end{proposition}

\begin{proof}  
The dynamics of the piecewise linear map $\Phi_{0}$ is particularly simple. The only possible fixed point of the map is $p_{0}$. As soon as $w^{*}\geq p_{0}$, this fixed point no longer exists, and the system displays a unique periodic orbit that contains $p_{0}$:
\begin{equation}\label{orbit}
\mathcal{O}_{p_0}=\{p_0,p_0+d, ..., p_0+(k-1)d, p_0, p_0+d, ..., \},
\end{equation}
where $k \geq 2$ is the smallest positive integer such that $w^*<p_0+(k-1)d$ and thus the orbit is periodic with period $k$.

It is easy to see that this orbit is globally attractive. Indeed, no orbit can be fully contained in the interval $(-\infty, w^{*})$ as $\Phi$ is above the identity line there. Thus any trajectory will eventually be mapped to $p_{0}$ and absorbed by the periodic orbit $\mathcal{O}_{p_0}$ given in~\eqref{orbit}. In other words, every orbit of this simple system is eventually periodic with period $p$. 
\end{proof}

Figure~\ref{fig:graphPhi0AndBifLinear} illustrates the shape of the map $\Phi_{0}$ and displays its periodic orbit as well as the associated bifurcations occurring under variation of $w^{*}$ or equivalently $v_{R}$.

\section{Persistence of period-incrementing in the non-singular case}\label{sec:PerAdd}

We work under assumptions~\ref{Assump:Converge},  \ref{Assump:NoSingPt} and \ref{Assump3} and build upon  results obtained in the limit $\eps\to 0$ to show that bidimensional integrate-and-fire neurons governed by \eqref{eq:SubthreshDyn} with $\eps>0$ undergo a  period-incrementing cascade as $v_{R}$ is increased. The main distinction between the singular limit and the system with $\eps>0$ is the fact that the adaptation map $\Phi_{\eps}$ is continuous and therefore the transitions from $p$- to $(p+1)$-periodic behavior are not instantaneous, leaving room for much richer dynamics. 
Although $\Phi_{\eps}$ is continous and unimodal, 
and the values of $w^{*}$ and $\Phi(w^{*})$ increase with $v_{R}$, the distance $ \Phi(w^*) - w^*$ does not necessarily vary monotonically  with $v_{R}$. Therefore, the adaptation map does not follow the same bifurcation pattern as the logistic map. 

Note that henceforth in this section, to lighten the notational burden, we will omit the index $\varepsilon$ indicating $\eps$-dependence except when particular emphasis on $\eps$ is important.

\subsection{Attracting periodic orbits}
It is known that necessarily, if $\Phi(w^*)\leq w^*$, then there is a globally attracting fixed point $w_f\in (-\infty,w^*]$ (see~\cite{touboul-brette:09}). In this section, we shall thus concentrate on the case where $\Phi(w^*)>w^*$.  In this case the unique fixed point $w^f$ of $\Phi$ is located in $(w^*,\Phi(w^*))$ and one can justify that the interval $[\Phi^2(w^*),\Phi(w^*)]$ is invariant and every trajectory enters this interval after at most a few iterates. Therefore, if additionally $\Phi^2(w^*)\geq w^*$, then the dynamics is trivial: $\Phi$ is strictly decreasing over $[\Phi^2(w^*),\Phi(w^*)]$ and thus either the fixed point $w^f\in [\Phi^2(w^*),\Phi(w^*)]$ attracts every trajectory  or (when $w^f$ is unstable) every trajectory (except the singleton $\{w^f\}$) tends to the period-2 orbit $\{ \Phi^2(w^*),\Phi(w^*), \ldots\}$. Thus we concentrate on the situation where
\begin{equation} \label{eq:phi2phi}
\Phi^2(w^*)<w^*<\Phi(w^*).
\end{equation}

We will often make the assumption that the fixed point $w^f$ is unstable, i.e. when $\Phi^{\prime}(w^f)<-1$. In this case,

\begin{equation}\label{defxi}
\xi:=\sup\{w\in [w^*,\Phi(w^*)]: \Phi^{\prime}(w)\leq -1\}
\end{equation}
is well-defined.

\begin{remark} Let us already mention that in Lemma \ref{ShapeLemma} below we show that the assumption \eqref{eq:phi2phi} holds for $\eps$ sufficiently small.
\end{remark}

Before we study the periodic solutions of $\Phi$, we need one more simple but useful result.
\begin{lemma}\label{LemmaAboutDer} The derivative of $\Phi$ satisfies the following estimates:
\begin{displaymath}
0<\Phi^{\prime}(w)<1 \quad \textrm{for} \ w<w^*. 
\end{displaymath}
\end{lemma}
\begin{proof} Recall from property P1 in Section \ref{sec:Model} that $\Phi$ is concave (with negative second derivative) in $(-\infty,w^*)$, therefore $\Phi^{\prime}$ is strictly decreasing therein. As $\Phi^{\prime}(w^*)=0$, it follows that $\Phi^{\prime}(w)>0$ for $w<w^*$. Similarly, if there was a point $\hat{w}$ such that $\Phi^{\prime}(\hat{w})=1$, then $\Phi^{\prime}>1$ for $w<\hat{w}$, which contradicts property P5, $\Phi(w) \geq w + d \geq w$ for $w<w^{**}$.  Therefore, $0<\Phi^{\prime}(w)<1$ for $w<w^*$. 
\end{proof}

We now provide sufficient conditions for the existence of attracting periodic orbits. We associate to any $k$-periodic orbit not containing the point $w^*$ a signature (or itinerary) as follows: if the orbit contains $m$ points smaller than $w^*$ followed by $n=k-m$ points larger than $w^*$, its itinerary is denoted $\mathcal{L}^{m}\mathcal{R}^n$.

 \begin{figure}[htbp]
	\centering
		\includegraphics[width=8cm]{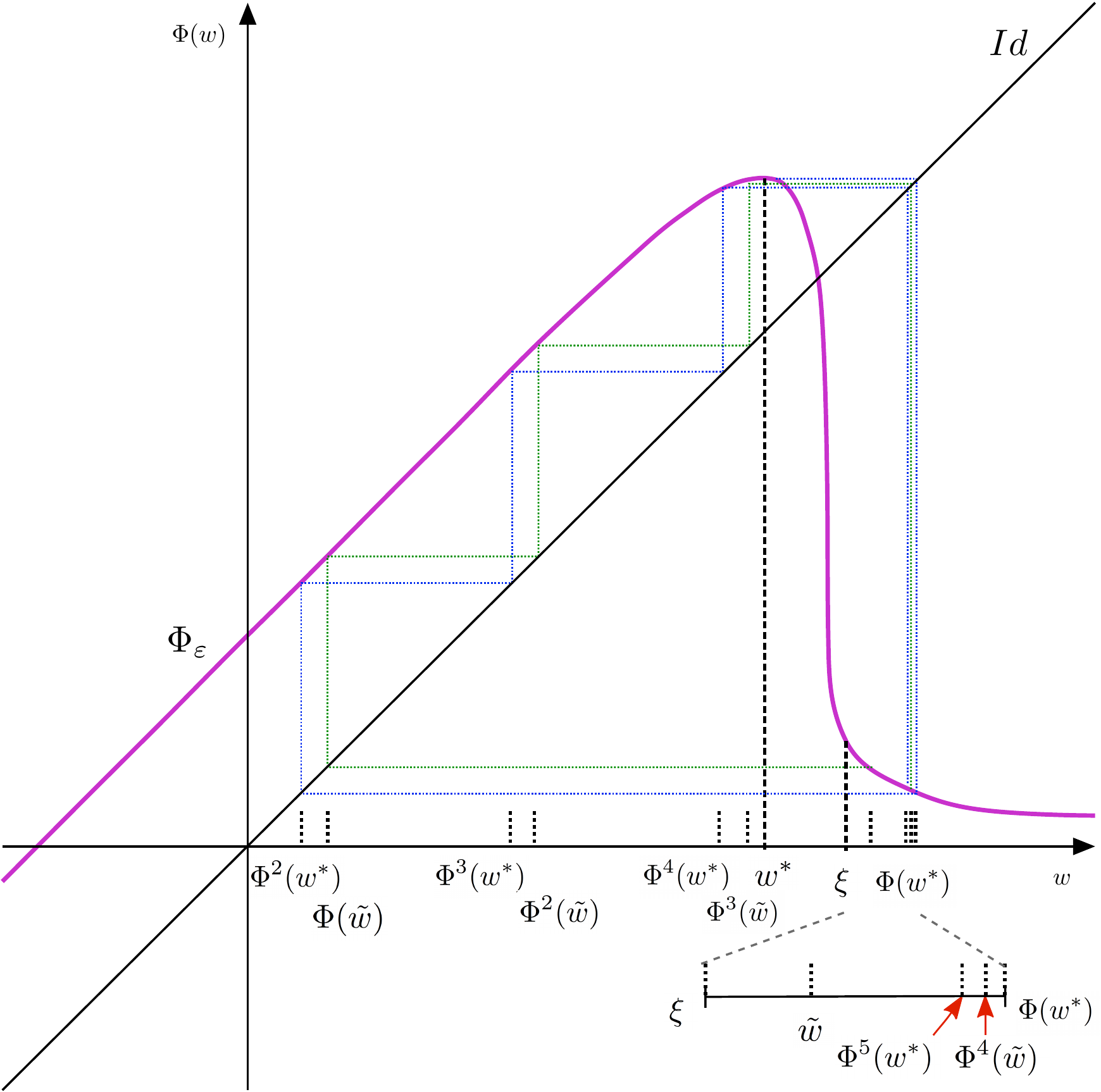}
	\caption{Illustration for the proof of Proposition \ref{StablePeriodk} with $k=4$, together with the orbits of $w^*$ and $\tilde{w}$ (see the statement of Proposition \ref{StablePeriodk}).}
	\label{fig:Picture_proof}
\end{figure}

\begin{proposition}\label{StablePeriodk} 
As discussed above, we assume that $w^*<w^f<\xi<\Phi(w^*)$ and $\Phi^{2}(w^{*})<w^{*}$.  If, moreover, $\Phi^3(w^*)<w^*$,  then let $k\in \N$ be defined as
\[ k= \min\{i\geq 3: \Phi^{i+1}(w^*)>w^*\}.\] 
If there exists $\tilde{w}\geq \xi$ such that (see Fig.~\ref{fig:Picture_proof}):
\begin{equation}\label{eq:OrbitOrder} 
	\Phi^i(w^*)< \Phi^{i-1}(\tilde{w})<w^*,  \ i\in\{2,3,...,k\}, \ \ \textrm{and} \ \Phi^{k+1}(w^*)>\tilde{w},
\end{equation}
then $\Phi$ admits an asymptotically stable $k$-periodic orbit,  with itinerary $\mathcal{L}^{k-1}\mathcal{R}^1$, attracting the orbit of $w^*$. Moreover, there is no other periodic orbit fully contained in the set $(-\infty,w^*]\cup [\tilde{w},\infty)$.  
\end{proposition}

 \begin{proof} Under these assumptions $\Phi^k$ maps the interval $[\tilde{w},\Phi(w^*)]$ onto
 \[
 [\Phi^{k+1}(w^*),\Phi^k(\tilde{w})]\subset [\tilde{w},\Phi(w^*)].
 \] 
 Thus interval $[\tilde{w},\Phi(w^*)]$ is $\Phi^k$ invariant and contains a fixed point for map $\Phi^k$. Because of assumption~\eqref{eq:OrbitOrder}, the $\Phi$-orbit of this $k$-periodic point admits $k-1$ points to the left of $w^*$ and one point to the right of $\xi$ (where the derivative of $\Phi$ is smaller than $1$ in absolute value). Hence, the derivative $\Phi^{\prime}$ along the orbit is negative and strictly greater than $-1$ on the account of Lemma \ref{LemmaAboutDer}, which proves that the associated fixed point of $\Phi^k$ is stable. 
 
Now suppose that there is another periodic orbit fully contained  in the set $(-\infty,w^*] \cup [\tilde{w},\infty)$. Then this orbit necessarily contains points belonging to $(-\infty,w^*]$ as well as some points of interval $[\tilde{w},\infty)$. But the points of this orbit in $[\tilde{w},\infty)$ are necessarily within  the interval $[\tilde{w},\Phi(w^*)]$, which is invariant for $\Phi^k$.   For any $w\in [\tilde{w},\Phi(w^*)]$ and $0<l<k$, $\Phi^l(w)\not\in [\tilde{w},\Phi(w^*)]$.  Thus, any periodic point in $(\tilde{w},\Phi(w^*))$ would necessarily be of period $nk$ for some $n\in\mathbb{N}$. Since $\Phi^{nk}$ is a contraction on $(\tilde{w},\infty)$, $\Phi^{nk}$ admits a unique fixed point in $(\tilde{w},\Phi(w^*))$. This fixed point is necessarily the unique fixed point of $\Phi^{k}$ in this interval that we identified before.
 \end{proof}

This result is quite general and will prove particularly useful for exhibiting the bifurcation structure of $\Phi$ for $\varepsilon$ small enough: it will guarantee the existence of an attractive $k$-periodic orbit. 
It also allows proving that actually most points, in a certain sense, are attracted by this $k$-periodic orbit: 
 \begin{corollary}\label{StablePeriodkAssympt}
 Grant the assumptions of Proposition  \ref{StablePeriodk} and define
 \[H:=A_1\cup A_2 \cup ... \cup A_{k-1},\]
 with $A_1:=(\gamma,\tilde{w})$, $\gamma:=\Phi^{-1}(\tilde{w})\cap (w^*,\Phi^{*}(w))$ and $A_i:=\Phi^{-1}(A_{i-1})\cap (\Phi^2(w^*),w^*)$, $i=2,...,k-1$.
The orbits of all points $w\in [\Phi^2(w^*),\Phi(w^*)]\setminus H$ are attracted by the $k$-periodic orbit given by Proposition  \ref{StablePeriodk}

\end{corollary}

\begin{proof}
The proof is elementary once we note that every point $w\in  [\Phi^2(w^*),\Phi(w^*)]\setminus H$ will after (at most) a few iterates enter the interval $[\tilde{w},\Phi(w^*)]$, and therefore 
any $w\in  [\Phi^2(w^*),\Phi(w^*)]\setminus H$ will be eventually attracted by the $k$-periodic orbit characterized in Proposition \ref{StablePeriodk}.
\end{proof}

\begin{remark} The assumption~\eqref{eq:OrbitOrder} for $k\geq 3$ of Proposition \ref{StablePeriodk} implies in particular that 
$\Phi^2(w^*)<\Phi^3(w^*)<w^*<\Phi(w^*)$. As proved below in Theorem \ref{NonalternatingChaos}, this inequality implies the existence of periodic orbits of all periods as well as chaos, at least on a compact subset of the dynamical core $[\Phi^2(w^*),\Phi(w^*)]$.
\end{remark}

Given these results, a natural question is whether the attracting $k$ periodic orbit in the Proposition \ref{StablePeriodk} is the only attracting periodic orbit of the map $\Phi$. We remark than in general for unimodal maps, this need not be the case. However, we know that any other such an attracting periodic orbit will not attract the critical point $w^*$ and will not have a point in the interval $[\tilde{w},\Phi(w^*)]$.

To go beyond this result and address the question of the uniqueness of the stable periodic orbit, we would like to make use of the theory of negative Schwarzian derivative maps. We recall that: 

\begin{definition} The Schwarzian derivative of a $C^3$ interval map $f:\mathcal{I}\to \mathcal{I}$ at $x\in \mathcal{I}$ such that $f^{\prime}(x)\neq 0$ is given by:

\begin{equation}\label{schwarzian}
(\mathrm{S}f)(x):=\frac{f^{\prime\prime\prime}(x)}{f^{\prime}(x)}-\frac{3}{2}\left(\frac{f^{\prime\prime}(x)}{f^{\prime}(x)}\right)^2.
\end{equation} 
Moreover, if a unimodal function $f$ has a unique critical point $c$ in $I$ (i.e. $f^{\prime}(c)=0$ and $f^{\prime}(x)\neq 0$ for $x\neq c$) and $(\mathrm{S}f)(x)<0$ for all $x\neq c$, then we say that $f$ is an $S$-unimodal map or, equivalently, that $f$ has negative Schwarzian derivative, which we denote $\mathrm{S}f<0$.  
\end{definition}

In our case $\mathrm{S}\Phi$ is well-defined everywhere except for  $w^*$, which is the only critical point, as we prove later in Theorem \ref{uniq-crit}. 

 \begin{corollary}\label{cor:other} Under the assumptions of Proposition \ref{StablePeriodk} suppose that  $\mathrm{S}\Phi<0$. Then the $k$-periodic orbit established by Proposition \ref{StablePeriodk} is the unique attracting periodic orbit  of $\Phi$.
 \end{corollary}
 
The above follows immediately from the Singer Theorem (see e.g. \cite{deMeloStrien_book}) as the $S$-unimodal map $f:\mathcal{I}\to \mathcal{I}$ with  no attracting periodic points in $\partial \mathcal{I}$ can have at most one attracting periodic orbit, i.e. the one which attracts the critical point. 

Since we aim to keep our analysis general, we do not restrict our choice of $F$ to achieve $\mathrm{S}\Phi<0$.  Nonetheless, this property certainly holds for certain choices of $F$ satisfying the  general regularity properties that we require and for certain corresponding parameter sets.  Thus, 
 we will assume $\mathrm{S}\Phi<0$ at some points to show the interested reader that when this condition holds, some of our results can be immediately strengthened.  

\subsection{Period-incrementing cascade}

Proposition~\ref{StablePeriodk} remains quite abstract at this level of generality, and a natural question that arises is to characterize the parameter sets for which the proposition applies, and to what extent the conditions of this proposition are satisfied. We will prove that as $\eps\to 0$, the proposition applies for almost all reset values $v_R$ and as $v_R$ increases provides a sequence of asymptotically-stable $k$-periodic orbits with incrementing $k$ that account for the numerically observed period-incrementing behavior (e.g., Figure \ref{fig:Bif}). 
 \begin{theorem}\label{new_incrementing}\textbf{[Period-incrementing]}
 For any integer $N>3$, there exist $\tilde{\varepsilon}>0$ and a sequence $\{{J}_k\}_{k=3}^{N}$  of ordered intervals ${J}_k$ of reset values $v_R$ (i.e. $v_{R_k}<v_{R_{k+1}}$ for any $v_{R_k}\in {J}_k$ and $v_{R_{k+1}}\in {J}_{k+1}$) such that for any $\varepsilon\leq \tilde{\varepsilon}$ and $v_R\in {J}_k$, $k=3,...,N$, the adaptation map $\Phi_{\varepsilon}$  has an asymptotically stable $k$-periodic orbit with itinerary $\mathcal{L}^{k-1}\mathcal{R}$.

 Furthermore, for any $\zeta>0$, we can pick $\tilde{\varepsilon}$ small enough so that for every $\eps \leq \tilde{\eps}$ and any $v_{R}\in {J}_k$ with $k\in \{3,\cdots,N\}$, the set $H_{\varepsilon}$ of initial conditions $w$ that might not be attracted by the $k$-periodic orbit of $\Phi_{\varepsilon}$ has Lebesgue measure smaller than $\zeta$.
 \end{theorem}
 \begin{remark} This phenomenon is clearly visible in our numerical simulations. Indeed, while with the increase of $v_R$ the period-incrementing sequence persists, the intervals of values of $v_{R}$ associated to complex and seemingly chaotic transitions between intervals of periodic behaviors diminish significantly; see Fig.~\ref{fig:Bif}.
 \end{remark}

Before we prove Theorem~\ref{new_incrementing}, we need a preliminary result relating  $\Phi_{\varepsilon}$ to $\Phi_0$.  
The results of Section~\ref{sec:SlowFast} ensure that $\Phi_{\varepsilon}$ approaches $\Phi_{0}$ in the Hausdorff distance. The object of the following Lemma is to establish the $C^{1}$-convergence of $\Phi_{\varepsilon}$ to $\Phi_{0}$ away from $w^{*}$.  Henceforth, we introduce the notation $w^*_{v_R}$ to emphasize that the value of $w^*$ is determined by the choice of $v_R$.
 \begin{lemma}\label{C1Lemma} Given any bounded interval $[v_{R_1},v_{R_2}]$ such that $v_{R_1}>v_F$,
 \begin{multline}
 \forall \delta >0 \ \forall \nu>0 \ \exists \tilde{\varepsilon}>0 \; \mbox{such that} \; \forall \varepsilon\leq\tilde{\varepsilon} \\
 	\forall v_R\in [v_{R_1},v_{R_2}] \ \forall w\in (-\infty, w^{*}_{v_R}-\nu]\cup [w^{*}_{v_R}+\nu,\infty), \\
 		\vert(\Phi_{\varepsilon})^{\prime}(w)-(\Phi_{0})^{\prime}(w)\vert<\delta. \nonumber
 \end{multline}
 \end{lemma}
 \begin{proof} The proof will use similar methods as that of  Proposition~\ref{limiting case}. First we prove the convergence in the interval $(-\infty, w^{*}_{v_R}-\nu]$. In any interval of this form we have $(\Phi_{0})^{\prime}(w)=1$, while $(\Phi_{\varepsilon})^{\prime}(w)$ satisfies 
 \begin{equation} \label{eq:Phiprime}
 (\Phi_{\varepsilon})^{\prime}(w)=\exp\left(\int_{v_R}^{\infty}\frac{\varepsilon (bu-F(u)-I)\;du}{(F(u)-W_{\varepsilon}(u;v_R,w)+I)^2}\right)
 \end{equation}
by an application of Peano's Theorem (see e.g. Theorem 3.1 and Corollary 3.1 in chapter 5 of \cite{phartman}). 
 We want the above to be close to $1$ and thus we need to show that the absolute value of the integral above is close to $0$.  But this in turn is equal to the sum of the integrals
 \[
 \int_{v_R}^{\breve{v}}\frac{\varepsilon (F(u)+I-bu)\;du}{(F(u)-W_{\varepsilon}(u;v_R,w)+I)^2}+\int_{\breve{v}}^{\infty}\frac{\varepsilon (F(u)+I-bu)\;du}{(F(u)-W_{\varepsilon}(u;v_R,w)+I)^2},
 \]
 where $\breve{v}$ denotes \sout{as before} the value of $v$ at which the solution $W_{\varepsilon}(v;v_R,w)$ intersects the $w$-nullcline (for initial conditions $w<w^{**}_{v_R}$, while for $w\geq w^{**}_{v_R}$ we can just take $\breve{v}=v_R$). Although $\breve{v}$ depends on $\varepsilon$ and $v_R$ it can be always overestimated as $\breve{v}<w^*_{v_{R_2}}/b$. Thus for the first integral above we have
 \[
 \begin{split}
& \int_{v_R}^{\breve{v}}\frac{\varepsilon (F(u)+I-bu)\;du}{(F(u)-W_{\varepsilon}(u;v_R,w)+I)^2}<\int_{v_R}^{\breve{v}}\frac{\varepsilon (F(u)+I-bu)\;du}{(F(v_R)+I -w)^2}<\\
& < \frac{\varepsilon}{\nu^2}\int_{v_R}^{\breve{v}}(F(u)+I-bu)\;du <\frac{\varepsilon}{\nu^2}\int_{v_{R_1}}^{w^*_{v_{R_2}}/b}(F(u)+I-bu)\;du< \frac{\varepsilon L_1}{\nu^2},
 \end{split}
 \]
 where $L_1$ is the value of the integral of $0< F(u)+I-w$ over the interval $[v_{R_1},w^*_{v_{R_2}}/b]$. For the second integral we compute
\[
\begin{split}
&  \int_{\breve{v}}^{\infty}\frac{\varepsilon (F(u)+I-bu)\;du}{(F(u)-W_{\varepsilon}(u;v_R,w)+I)^2}<\varepsilon\int_{\breve{v}}^{\infty}\frac{\;du}{F(u)-W_{\varepsilon}(u;v_R,w)+I}<\\
& < \varepsilon \int_{\breve{v}}^{\infty}\frac{\;du}{F(u)+I -bu}<\varepsilon \int_{v_{R_1}}^{\infty}\frac{\;du}{F(u)+I -bu}<\varepsilon L_2,
 \end{split}
 \]
 where $L_2$ denotes the value of the (convergent) integral of $1/(F(u)+I -bu)$ on the interval $(v_{R_1},\infty)$.
 In this way we have obtained the desired estimate independently of $v_R\in [v_{R_1},v_{R_2}]$; to complete the proof for $w\in (-\infty, w^*_{v_R}-\nu]$, it is sufficient to chose 
$\tilde{\varepsilon}<\frac{\delta}{\frac{L_1}{\nu^2}+L_2}.$

 To show the convergence  for $w\in [w^*_{v_R}+\nu,\infty)$, we notice that  on this domain, the adaptation map can be expressed as the composition of two maps,
 \[
 \Phi_{\varepsilon}(w)=\Phi_{\varepsilon}(\Theta_{\varepsilon}(w)),
 \]
 where  $\Theta_{\varepsilon}$ assigns to $w>w^*_{v_R}$ the point $\hat{w}$ on the reset line $v=v_R$, below the $v$-nullcline, such that the trajectory $(V_{\varepsilon}(t;v_R,w),W_{\varepsilon}(t;v_R,w))$ crosses the reset line at this point before spiking. Thus
  \begin{equation}\label{derivComp}
( \Phi_{\varepsilon})^{\prime}(w)=(\Phi_{\varepsilon})^{\prime}(\Theta_{\varepsilon}(w))(\Theta_{\varepsilon})^{\prime}(w).
 \end{equation}
The repulsive slow manifold $\mathcal{C}^+_{\varepsilon}$ (prolonged by the flow) intersects  $\{ v=v_R \}$ both above and below the $v$-nullcline.  Having computed $\tilde{\varepsilon}$ for the first part of the proof, we can further assume that $\tilde{\varepsilon}$ is so small that $\varepsilon\leq \tilde{\varepsilon}$ implies that for every $v_R\in [v_{R_1},v_{R_2}]$, $w^*_{v_R}+\nu$ lies above the intersection of  $\{ v=v_R\}$ with $\mathcal{C}^+_{\varepsilon}$ above the $v$-nullcline.  As a result, we can conclude that $\Theta_{\varepsilon}(w)$ intersects $\{ v=v_R\}$ below the lower intersection of $\mathcal{C}^+_{\varepsilon}$ with this line.

 Since  the first factor $(\Phi_{\varepsilon})^{\prime}(\Theta_{\varepsilon}(w))$ in the multiplication in the formula (\ref{derivComp}) is always non-negative and bounded by $1$, in order to show that $(\Phi_{\varepsilon})^{\prime}(w)$ is small, it suffices to show that
 $
\vert (\Theta_{\varepsilon})^{\prime}(w)\vert\approx 0
 $
 for small $\varepsilon$. For this purpose, consider initial conditions $(v_R,w_1(0)), (v_R,w_2(0):=w_1(0)+\Delta w)$ with $w_1(0)>w_{v_R}^*+\nu$, such that for both orbits, $v'(0),w'(0)<0$.
Denote trajectories from these initial conditions by $\gamma_k(t)=(v_k(t),w_k(t)):=(V_{\eps}(t;v_R,w_k),W_{\eps}(t;v_R,w_k))$ for $k=1,2$. Note that each $\gamma_k$ intersects the $v$-nullcline with $v<v_R$.
Define a section $\Sigma$ of constant $v$, say with $v=v_{\Sigma}$, transverse to the flow from $\{ (v_R,w) : w_1(0) \leq w \leq w_2(0) \}$, such that $\gamma_1, \gamma_2$ intersect $\Sigma$ at times $t_1, t_2$ respectively, with the $\gamma_1$ intersection lying an $O(\varepsilon)$-distance from the $v$-nullcline.
We have $ |w_1(t_1)-w_2(t_2)| = c_1(\eps)\Delta w$, where $c_1(\eps) \to 1$ as $\eps \to 0$.

Now, after these intersections, $\gamma_1$ is bounded between $C_{\eps}^-$, namely the invariant slow manifold that perturbs from the attracting branch of the critical manifold for $0 < \eps \ll 1$, and the critical manifold itself, while $\gamma_2$ is bounded between $C_{\eps}^-$ and $\gamma_1$.
Each trajectory crosses the $v$-nullcline and returns to intersect $\Sigma$ a second time. Since the trajectories only traverse an $O(\eps)$ distance between intersections with $\Sigma$, the flow box theorem ensures that their $w$-coordinates differ by $c_2(\eps)\Delta w$ for an $O(1)$ constant $c_2(\eps)$ at the second intersection.

Finally, we invoke the slow divergence integral \cite{dumortier1996,deM} to quantify the change in the distance between the trajectories' $w$-coordinates as they evolve from $\Sigma$ back to $\{ v=v_R \}$ below the $v$-nullcline (and $C_{\eps}^+$), close to $C_{\eps}^-$.  Let $(v_{\Sigma},w_{\eps}^-)$ denote the intersection of $C_{\eps}^-$ with $\Sigma$ and $W(v; v_{\Sigma},w_{\eps}^-)$ the $w$-coordinate of $C_{\eps}^-$ expressed as a graph of $v$ for $v \geq v_{\Sigma}$.
In our case, the slow divergence integral is given by 
\[
\int_{v_{\Sigma}}^{v_R} \frac{\textstyle (F'(v))^2}{\textstyle bv - W(v; v_{\Sigma}, w_{\eps}^-)} \, dv. 
\]
Assume that $v_R$ is such that $w^{**}<w_F$. Then we can take $\eps$ sufficiently small such that $C_{\eps}^-$ is bounded away from the $w$-nullcline between $\Sigma$ and $\{ v=v_R \}$.
Then this integral is negative and $O(1)$, such that we have an $O(1)$ exponential contraction in $w$  from $\Sigma$ to $\{ v=v_R \}$.  Hence,  $(\Theta_{\eps})'(w)$ can be made arbitrarily small by shrinking $\eps$,  as desired.  This completes the proof for $w^{**}<w_F$.  If $w^{**} \geq w_F$, then yet another section is needed, say at $\{v=\tilde{v}_R\}$ with $w^{**}_{\tilde{v}_R}<w_F$.  The previous arguments bound the expansion from $\{ v=\tilde{v}_R \}$ to $\{ v =v_R \}$ and hence the desired result still holds.
 
 \end{proof}

 \begin{proof} [Proof of Theorem \ref{new_incrementing}] The proof relies on the fact that for the map $\Phi_0$ it is relatively easy to satisfy the assumptions of Proposition \ref{StablePeriodk}; given this observation, we can then exploit the closeness of  $\Phi_{\varepsilon}>0$ to $\Phi_0$.

 Indeed, given $N>3$, there exists a sequence ${J}_k^0$,  $k=3,...,N$, such that for any $v_R\in J^0_k$ the orbit of $p_0$ under $\Phi_0$ is $k$-periodic, with itinerary $\mathcal{L}^{k-1}\mathcal{R}$. Since the set $\bigcup_{k=3}^{N} {J}_k^0$ is bounded,  we conclude that for any $\delta>0$ there exists $\tilde{\varepsilon}>0$ such that for any $\varepsilon\leq \tilde{\varepsilon}$, for any  $k\in\{3,...,N\}$ and any $v_R\in J_k^0$ our convergence results imply:
 \[
 \forall w\in (-\infty, w^*_{v_R}] \cup [w^*_{v_R}+\delta/2,\infty) \  \forall i\in \{0,1,...,k+1\},  \quad \vert (\Phi_{\varepsilon})^i(w)-(\Phi_{0})^i(w)\vert <\frac{\delta}{2}
 \]
  and simultaneously
 \begin{equation} \label{eq:c1_conv}
 \forall w\in [w^*_{v_R}+\delta/2,\infty), \quad  \vert (\Phi_{\varepsilon})^{\prime}(w)-(\Phi_0^{v_R})^{\prime}(w)\vert<\frac{\delta}{2}.
 \end{equation}

We henceforth assume $\delta<\min \{1,d/(N+1) \}$, thus the above implies that $\vert(\Phi_{\varepsilon})^{\prime}(w)\vert<1/2<1$ for every $w\in [w^{*}_{v_R}+\delta/2, \infty)$ (and every choice of $v_R\in J_k^0$).  Moreover, for any $\delta<1$ chosen, one can always, if necessary, instead of taking ``large" intervals $J^0_k$ (of the length equal to $d$), take their subintervals (which will be for simplicity denoted again $J^0_k$) so that with any $v_R\in J_k^0$ we have
 \[
 p_0+(k-2)d=(\Phi_0)^{k-2}(p_0)\in (w^*_{v_R}-\frac{3\delta}{2}, w^*_{v_R}-\frac{\delta}{2})
 \]
 which implies that
  \[
  p_0+(k-1)d=(\Phi_0)^{k+1}(w^*_{v_R})=(\Phi_0)^{k-1}(p_0)>w^*_{v_R}-\frac{3\delta}{2}+d.
 \]
Now having $\varepsilon\leq \tilde{\varepsilon}$, for any $k\in\{3,...,N\}$, let $J_k:=J_k^0$.  Note that by assumption $\delta < d/(N+1) < 2d/5$, and therefore the interval $(w^*_{v_R}+\delta/2,w^*_{v_R}+d-2\delta)$ is not empty.   Now, for arbitrary $v_R\in J_k$ and arbitrary $\tilde{w}_{v_R}\in (w^*_{v_R}+\delta/2,w^*_{v_R}+d-2\delta)$  we have  $\vert(\Phi_{\varepsilon})^{\prime}(\tilde{w}_{v_R})\vert<1$ by (\ref{eq:c1_conv}) and hence $\xi^{v_R}_{\varepsilon}<\tilde{w}_{v_R}$ (where $\xi^{v_R}_{\varepsilon}$ denotes the point $\xi$ as defined in~\eqref{defxi} for particular values of $\eps$ and $v_R$), and also $\tilde{w}_{v_R}<\Phi_{\varepsilon}(w^*_{v_R})$  since 
 \[
 \Phi_{\varepsilon}(w^*_{v_R})\in (w^*_{v_R}+d-\delta/2,w^*_{v_R}+d+\delta/2). 
 \]

 With the above choice of $\tilde{w}_{v_R}$,
 \[
(\Phi_{\varepsilon})^{k+1}(w^*_{v_R})>\tilde{w}_{v_R} \quad \textrm{and} \quad  (\Phi_{\varepsilon})^{k-1}(\tilde{w}_{v_R})<w^*_{v_R}
 \]
as
\[
\vert (\Phi_{\varepsilon})^{k-1}(\tilde{w}_{v_R})-(\Phi_0)^{k-1}(\tilde{w}_{v_R})\vert <\delta/2 \quad \textrm{and} \quad (\Phi_{0})^{k-1}(\tilde{w}_{v_R})<w^*_{v_R}-\delta/2
\]
We conclude that the $v_R$-dependent maps $\Phi_0$ and $\Phi_\eps$ satisfy the assumptions of Proposition \ref{StablePeriodk} for any $v_R\in J_k$, where $k \in \{3,..,N\}$, and the maps $\Phi_{\varepsilon}$ have $k$-periodic orbits of signature $\mathcal{L}^{k-1}\mathcal{R}$, asymptotically stable. Notice that all the above reasoning stands for any $\delta\leq \min\{1,d/(N+1)\}$.

We now prove the second statement.  Choose $\zeta>0$ arbitrarily and select any $k\in \{3,...,N\}$ and $v_R\in J_k$. We define  
\[
\gamma:=(\Phi_{\varepsilon})^{-1}(\tilde{w}_{v_R})\cap (w^*_{v_R},\Phi_{\varepsilon}(w^*_{v_R}))
\]
and introduce the sets
\begin{eqnarray*}
A_1&:=& (\gamma,\tilde{w}_{v_R}), \\
A_i &:=& (\Phi_{\varepsilon})^{-1}(A_{i-1}) \cap ((\Phi_{\varepsilon})^2(w^{*}_{v_R}),w^*_{v_R}), \quad i \in \{2,...,k-1\}, \\
H   &:=&  \bigcup_{i=1}^{k-1} A_i,
\end{eqnarray*}
which are well-defined for $\eps$ sufficiently small. 
As from the first part of the proof $\tilde{w}_{v_R}\in (w^{*}_{v_R}+\delta/2, w^{*}_{v_R}+d-2\delta)$ can be arbitrary, we can assume that $\tilde{w}_{v_R}<w^*_{v_R}+\delta$, so that $\mathrm{diam}A_1<\delta$. Simultaneously, we see that the closer $\tilde{w}_{v_R}>w^*_{v_R}$ to $w^*_{v_R}$ (which corresponds to choosing $\delta$ smaller), the further (to the left) from $w^*_{v_R}$ is the right-endpoint of interval $A_2$. Let $B_2<w^*_{v_R}$ be the right-endpoint of interval $A_2$. Then:
 \[
 \forall i \in \{2,3,...N\}, \quad \mathrm{diam}A_i<\left(\frac{1}{\underset{w\in (-\infty, B_2]}{\inf}(\Phi_{\varepsilon})^{\prime}(w)}\right)^i\mathrm{diam}A_1.
 \]

 Although $(\Phi_{\varepsilon})^{\prime}(w)\to 0$ as $w\to w^*_{v_R}$, we have that $\inf_{w\in (-\infty, B]}(\Phi_{\varepsilon})^{\prime}(w)$ is isolated from $0$ in every interval $(-\infty,B]$ where $B< w^*_{v_R}$; that is,  there exists $M>0$ such that
 \[
  \forall i \in \{2,3,...N\}, \quad\mathrm{diam}A_i< M\mathrm{diam}(A_1) <  M\delta.
 \]
 Since we consider finite set of $k$ values ($k\in \{3,..,N\}$) and sets $J_k$ are bounded, there exists a constant $\tilde{M}\geq M$ that is  independent of the choice of $k$, $v_R\in J_k$ and any $\varepsilon\leq \tilde{\varepsilon}$ with $\tilde{\varepsilon}$  fixed in the first part of this proof for $\delta \leq \min\{1,d/(N+1)\}$. Note in particular that taking $\varepsilon$ smaller (e.g., by further lowering $\delta$) makes the derivative $(\Phi_{\varepsilon})^{\prime}(w)$ closer to $1$, thus further from $0$, for any fixed $w<w^*_{v_R}$ and therefore reinforces the above estimates. 

Thus we have the explicit expression
 \[
 \delta\leq \min\left\{1,\frac{d}{N+1},\frac{\zeta}{N\tilde{M}}\right\} \ \ \Longrightarrow \ \  \forall i \in \{2,3,...N\}, \ \  \mathrm{diam} A_i<\tilde{M}\delta<\frac{\zeta}{N}, 
 \]
which yields that for every $v_R\in J_k$ and $k \in \{3,4,..,N\}$, the Lebesgue measure of $H$ is bounded:
\[
\Lambda(H)\leq \sum_{i=1}^{k-1}\mathrm{diam}(A_i)<\frac{N\zeta}{N}=\zeta.
\]  
\end{proof}

\begin{figure}
\begin{center}
\includegraphics[width=.7\textwidth]{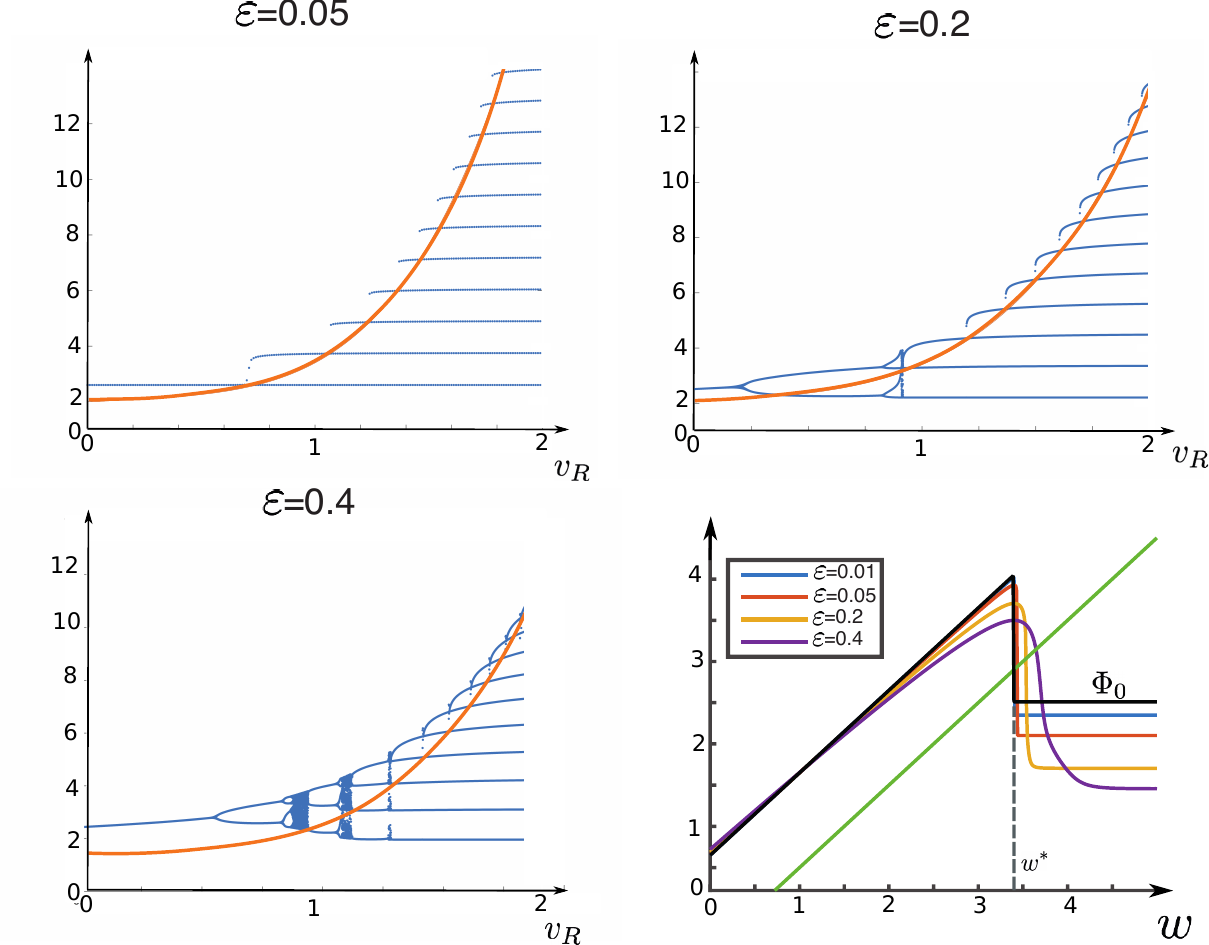}
\end{center}
\caption{Maps $\Phi_{\eps}$ and bifurcation sequences for the standard quartic model~\eqref{eq:StandardParams} and various values of $\varepsilon$.  For $\eps=0.05$, the diagram is very close to the bifurcation diagram of $\Phi_{0}$ depicted in Fig.~\ref{fig:graphPhi0AndBifLinear} (we superimposed the plot of the map $F(v)+I$ in orange to emphasize this similarity). 
As $\eps$ is increased, the map $\Phi$ slowly deviates from $\Phi_{0}$; the associated bifurcation diagrams conserve the overall period-incrementing structure, but with larger transition regimes characterized by the presence of chaos. We also note that as $v_{R}$ increases, the bifurcation structure becomes more similar to the singular limit.}
\label{fig:different_epsilons}
\end{figure}

Figure~\ref{fig:different_epsilons} illustrates the results of this section by showing the orbits, as a function of $v_{R}$, of the adaptation map in the case of the quartic model $F(v)=v^{4}+2av$, for different values of $\eps$. We clearly observe  the convergence of the diagram towards that of $\Phi_{0}$ as well as a number of results demonstrated in this section, particularly the fact that the region of parameter values for which the system has stable periodic orbits shrinks as $\eps$ increases.

\section{Chaos between period-incrementing transitions}\label{Chaos}

We have showed that the bifurcation diagram of the non-singular system shows a period-incrementing structure, in the sense that there exists a sequence of disjoint ordered intervals of values of  $v_{R}$  for which the adaptation map features attractive periodic orbits of incrementing periods. In this section, we focus on the phenomena arising between two intervals ${J}_k$ and ${J}_{k+1}$, i.e. at the transition between periodic orbits of periods $k$ and $k+1$.  
Chaotic period-incrementing transitions are expected in continuous maps well approximated by discontinous piecewise linear maps featuring pure period-incrementing (see e.g. \cite{pring}). 
Numerically, we observe chaos in transitions from period 2 to period 3, period 3 to 4, and possibly for additional period-incrementing transitions, 
 preceded by one or a few period-doubling bifurcations, provided that one takes $\varepsilon$ not too small (see Figure \ref{fig:different_epsilons}).   As a more detailed example,  in Fig.~\ref{fig:Transition23}, we numerically illustrate the transitions from bursts of period 2 to 3.  We observe that the period-2 orbit loses stability through a period-doubling bifurcation, yielding a period-4 orbit with points that progressively approach  the region of instability where $\Phi'$ is strictly smaller than $-1$; this period-4 orbit again loses stability, seemingly through a period-doubling bifurcation, and very rapidly progresses into a chaotic trajectory before suddenly stabilizing on a period-3  orbit. 
In fact, as we will see, chaos is present between any two intervals ${J}_k$ and ${J}_{k+1}$, $k\geq 2$, but for greater values of $v_R$ and hence larger $k$,  the chaotic transitions are more abrupt and therefore less visible in numerical simulations. 
\begin{figure}[h]
\centering
\includegraphics[width=.8\textwidth]{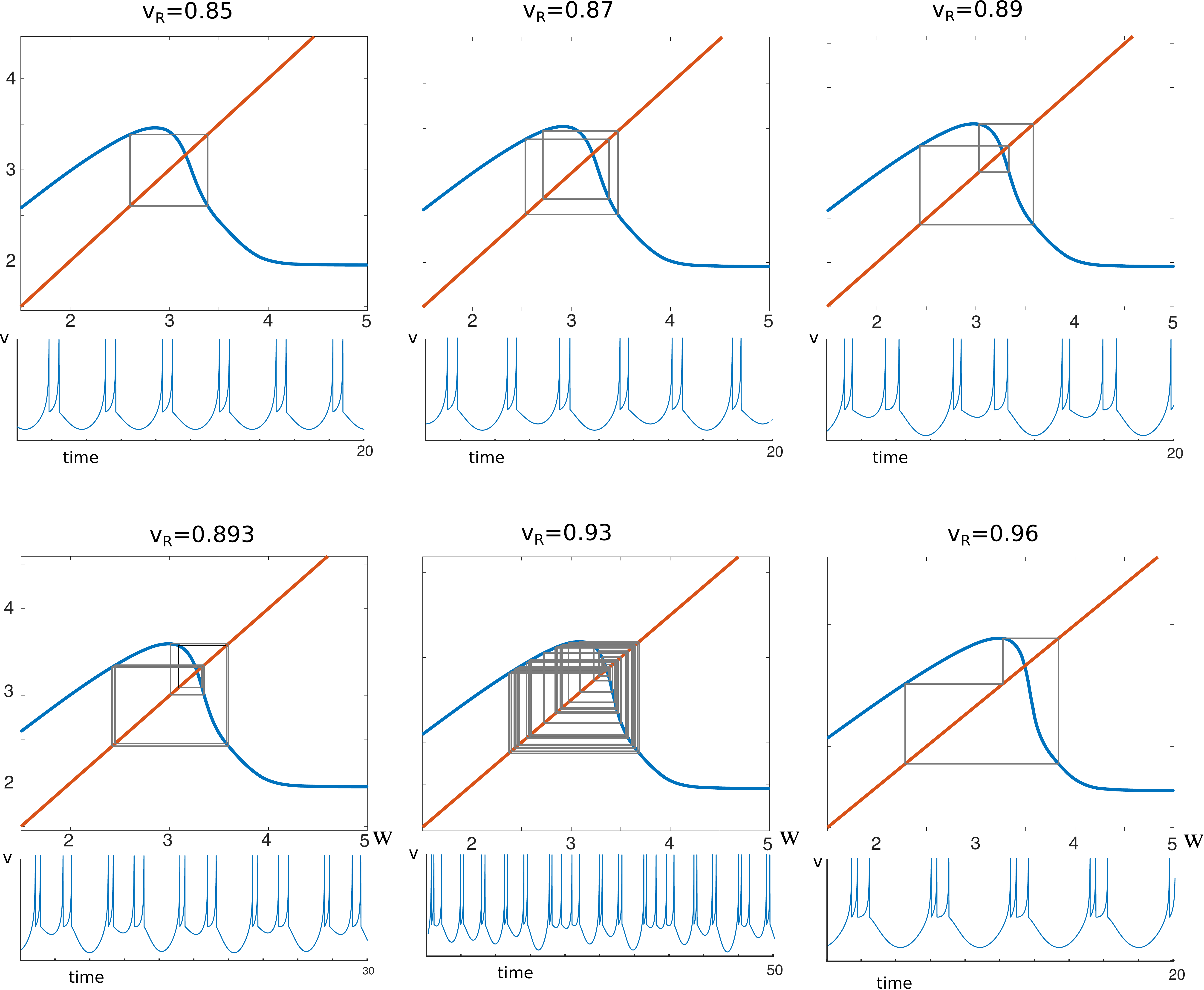}
\caption{Orbits of $\Phi$ for the standard quartic model~\eqref{eq:StandardParams} with $\eps=0.4$, for values of $v_{R}$ spanning the period-incrementing transition between bursts with 2 and 3 spikes. A period-2 orbit undergoes a period doubling giving rise to a period-4 orbit, which itself loses stability, yielding chaotic spiking, before the system stabilizes on a period-3 orbit. }
\label{fig:Transition23}
\end{figure}
Our main results of this section can be outlined as follows:

\emph{Topological chaos (including the existence of periodic orbits of all periods and sensitive dependence on initial conditions) occurs for all parameter values $v_R$ big enough (Theorem \ref{NonalternatingChaos}). Moreover, at the transition between periodic orbits of types $\mathcal{L}^{k-1}\mathcal{R}$ and  $\mathcal{L}^{k}\mathcal{R}$, one can expect a positive measure set of parameter values $v_R$ for which the system is strongly chaotic, with positive Lyapunov exponent  and, presumably, with absolutely continuous invariant probability measure.}

Before we  proceed, let us make also one simple remark concerning the observed period-doubling bifurcations.
Suppose that for some value $v_R$,  $\Phi$ has an attracting periodic orbit $\mathcal{O}$ of type $\mathcal{L}^{k-1}\mathcal{R}$  (e.g., when the assumptions of Proposition \ref{StablePeriodk} hold). Then for arbitrary  $w_0\in \mathcal{O}^{v_R}$, we have $-1<(\Phi^k)^{\prime}(w_0)<0$. Consequently, the orbit can lose its stability only when the derivative of $\Phi^k$ at points on the orbit reaches $-1$. This observation allows us to  predict the presence of period-doubling bifurcations in the transition to $\mathcal{L}^{k}\mathcal{R}$ orbits.

The purpose of this section is to rigorously explain why we necessarily observe chaotic behavior in such families of continuous maps and to describe precisely the chaotic properties of the adaptation map in corresponding regions of parameter space, using different notions of chaos. To this end, we will need to invoke some powerful results on the dynamics of unimodal maps, which, although today well-known to the specialist in the field, are certainly non-trivial.  In the first part \ref{ChaosGeneral} we establish some general properties of the adaptation map such as uniqueness of the critical point (Theorem \ref{uniq-crit}) and its non-degeneracy (Theorem \ref{non-degeneracy}). Next, in Section \ref{TopolChaos} we refer to various notions of topological chaos and show that for almost all parameter values $v_R$ the map $\Phi$ exhibits topological chaos, with any of the reasonable definitions of chaos, such as e.g. chaos in the sense of Devaney or Block and Coppel, all of them being equivalent for the adaptation map. However, since topological chaos is not always reflected in the observed behavior of the system,  in Section \ref{ChaosUnimod}  we aim to explain the occurrence of the chaos that is  clearly visible in our bifurcation diagrams. We use one more notion of chaos, namely metric chaos, saying \sout{roughly speaking}, that the map is chaotic when it admits an invariant measure, absolutely continuous with respect to the Lebesgue measure. This is one of the strongest notions  of chaos.

\subsection{A few additional useful properties of the adaptation map}\label{ChaosGeneral}

 We start  with a result that, for a range of $v_R$ values, establishes the non-monotonicity in the dynamical core of the adaptation map formed by the initial iterates of $w_{v_R}^*$.

\begin{lemma}\label{ShapeLemma}
Given $[v_{R_1},v_{R_2}]$, where $v_{R_1}>v_F$ is such that
\begin{equation}\label{condition}
w^*_{v_{R_1}}=F(v_{R_1})+I>p_0+ ld, \ \ \textrm{for some integer} \ l\geq2,
\end{equation}
there exists $\tilde{\varepsilon}$ such that for every $v_{R}\in [v_{R_1},v_{R_2}]$ and $\varepsilon\leq \tilde{\varepsilon}$, we have 
\begin{equation}\label{ShapeCondition}
(\Phi_{\varepsilon})^2(w^*_{v_R})< w^*_{v_R}<\Phi_{\varepsilon}(w^*_{v_R}).
\end{equation}
Moreover, if for given $\varepsilon\leq \tilde{\varepsilon}$ we have $\Phi_{\eps}^{\prime}(w)<-1$ for $w\in (\alpha,\xi)$, where
\begin{displaymath}
\xi:=\sup\{w\in [w^*,\Phi_{\eps}(w^*)]: \ \Phi_{\eps}^{\prime}(w)\leq -1\}, \quad \alpha:=\inf\{w\in [w^*,\Phi_{\eps}(w^*)]: \ \Phi_{\eps}^{\prime}(w)\leq-1\},
\end{displaymath}
then the fixed point $w^f\in (w^*_{v_R},\Phi_{\varepsilon}(w^*_{v_R}))$ is unstable, i.e. $(\Phi_{\varepsilon})^{\prime}(w^f)<-1$.
\end{lemma}
\begin{proof} Recall that $p_0:=w_F+d=F(v_F)+I+d$ and notice that if \eqref{condition} holds for some $v_{R_1}>v_F$ then it holds also for any $v_R>v_{R_1}$ (replacing $v_{R_1}$ with $v_R$). Moreover, \eqref{condition} implies that
\begin{equation}\label{ShapeCondition0}
(\Phi_{0})^2(w^*_{v_R})< w^*_{v_R}<\Phi_{0}(w^*_{v_R})
\end{equation}
for any $v_R\geq v_{R_1}$, as $\Phi_{0}(w^*_{v_R})=w^*_{v_R}+d$ and $(\Phi_{0})^2(w^*_{v_R})=p_0=w_F+d<w^*_{v_R}$. We know already that $\Phi_{\varepsilon}$ can be $C^0$- and $C^1$-approximated by $\Phi_0$ on appropriate intervals, uniformly in $v_R$ for $v_R\in [v_{R_1},v_{R_2}]$, by Proposition \ref{limiting case} and Lemma  \ref{C1Lemma}.

We recall that the maps $\Phi_{\eps}$ and $\Phi_{0}$ vary with $v_R$ (i.e.  $\Phi_{\eps}=\Phi_{\eps}^{v_R}$ and $\Phi_{0}=\Phi_{0}^{v_R}$) but  we omit  the index $v_R$ for clarity of  notation (it will be clear that all the estimates can be done uniformly for $v_R\in [v_{R_1},v_{R_2}]$). Without loss of generality, we can assume that $0<\delta \ll 1$ (in particular $\delta<d$, say $\delta<\frac{d}{3}$).  Then, for sufficiently small $\varepsilon$, we have:
\begin{displaymath}
\Phi_{\varepsilon}(w^*)>w^*+\frac{2}{3}d>w^*+\delta>w^*
\end{displaymath}
and $\Phi_{\varepsilon}^2(w^*)\in (w_F+d-\delta,w_F+d+\delta)$, which implies
\begin{displaymath}
 \Phi_{\varepsilon}^2(w^*)<w^*-(l-\frac{1}{3})d<w^*-\delta<w^*.
\end{displaymath}
Therefore \eqref{ShapeCondition} is satisfied, for every $v_R\in [v_{R_1},v_{R_2}]$. 

Since \[w^*+\frac{4}{3}d>w^*+d+\delta=\Phi_0(w^*)+\delta>\Phi_{\varepsilon}(w^*)>w^*+\frac{2}{3}d\] and 
\[(\Phi_{\varepsilon})^2(w^*)<w^*-d(l-\frac{1}{3})\leq w^*-\frac{5}{3}d,\] 
there necessarily exists $w\in (w^*,\Phi_{\varepsilon}(w^*))$ such that $(\Phi_{\varepsilon})^{\prime}(w)<-1$ (heuristically, the function has a sharp drop within a small interval). Therefore the points $\alpha$ and $\xi$ are well-defined for every  $\varepsilon\leq \tilde{\varepsilon}$. 
Moreover, as $\Phi_{\eps}$ is $C^1$-close to $\Phi_0$ on $[w^*+\delta,\infty)$, we necessarily have
\begin{displaymath}
w^*<\alpha<\xi<w^*+\delta <w^*+d
\end{displaymath}
with
\begin{displaymath}
\Phi_{\varepsilon}(\xi)<w^*<\xi
\end{displaymath}
(otherwise on the interval $(\xi,\Phi_{\varepsilon}(w^*))$ of length smaller than $4d/3$, $\Phi_{\varepsilon}$ shall decrease at least by $5d/3$,  which is impossible as $\Phi_{\varepsilon}^{\prime}(w)\geq -1$ for $w>\xi)$. Similarly, we can justify that $\Phi_{\varepsilon}(\alpha)>\alpha$. Indeed, as $\Phi_{\varepsilon}^{\prime}(w)\in (-1,0)$ for $w\in (w^*,\alpha)$, we have 
\begin{displaymath}
w^*+d-\delta-\Phi_{\varepsilon}(\alpha)<\Phi_{\varepsilon}(w^*)-\Phi_{\varepsilon}(\alpha)<\alpha-w^*
\end{displaymath}
which implies
\begin{displaymath}
\Phi_{\varepsilon}(\alpha)>w^*+d-\delta-\alpha+w^*>w^*+d-2\delta>w^*+\delta>\alpha.
\end{displaymath}
Since at the point $w=\alpha$ the graph of $\Phi_{\varepsilon}$ is above the identity line and at $w=\xi$ below, we necessarily obtain $w^f\in (\alpha, \xi)$ and the statement about instability of $w_f$ follows. 
\end{proof}
\begin{remark}\label{RemarkShapeLemma} Note that from the above proof it follows that any given condition of the form $\Phi^{2}(w^*)<\Phi^ {3}(w^*)<...<\Phi^k(w^*)<w^*<\Phi(w^*)$  is guaranteed to hold for $v_R$ sufficiently large to satisfy \eqref{condition}, for a large enough choice of $l$ and $\eps$ sufficiently small. 
\end{remark}
\begin{remark} We introduced here the extra technical assumption that $\Phi_{\eps}^{\prime}(w)<-1$ on $(\alpha,\xi)$; our numerical simulations suggest that this is always satisfied and that $\Phi$ has only one inflection point, located in $(\alpha,\xi)$.   
\end{remark}

\begin{definition}\label{DefNonDeg} We say that the critical point $w^*$ of $\Phi$ is \emph{non-degenerate} if $\Phi^{\prime\prime}(w^*)\neq 0$. 
\end{definition}

We recall that under the current assumptions, $\Phi$ is at least $C^3$ since it is given by the flow of  (\ref{eq:SubthreshDyn}) with $F$ being at least $C^3$. In fact, in the most common cases, such as the adaptive exponential model $F(v)=\mathrm{e}^v-v$ or the quartic model $F(v)=v^4+2av$, we can expect $\Phi\in C^{\infty}$.  Since the results below (Theorems \ref{uniq-crit} and \ref{non-degeneracy}) do not depend on $\eps>0$ but the dependence of $\Phi$ on $v_R$ is important in the proofs, we use the notation $\Phi_{v_R}$ for the adaptation map in the remainder of this subsection.

\begin{theorem}\label{uniq-crit} For every $v_R$ the point $w^*$ is the unique critical point of $\Phi_{v_R}$. 
\end{theorem}
\begin{proof}
We want to show that $\Phi^{\prime}_{v_R}(w)\neq 0$ for any $w\neq w^*$. For $w<w^*$ the statement follows immediately from Lemma \ref{LemmaAboutDer} (also from equation (\ref{eq:Phiprime})).  

Now take $w>w^*$. By $P_{v_R}(w)<w^*<w$ denote the $w$-coordinate of the crossing of the trajectory $(V(t;w,v_R),W(t;w,v_R))$ with the reset line $v=v_R$ (below the $v$-nullcline). We have two possibilities:
\begin{enumerate}
\item[(a)] $P_{v_R}(w)>bv_R$ (crossing above the $w$-nullcline), or 
\item[(b)] $P_{v_R}(w)\leq bv_R$ (crossing below the $w$-nullcline).
\end{enumerate}
In both cases we compute $\Phi_{v_R}(w)=\Phi_{v_R}(P_{v_R}(w))$ and $\Phi^{\prime}_{v_R}(w)=\Phi^{\prime}_{v_R}(P_{v_R}(w)) P^{\prime}_{v_R}(w)$, where $\Phi^{\prime}_{v_R}(P_{v_R}(w))\neq 0$ by the previous ($w<w^*$) result. It remains to show that $P^{\prime}_{v_R}(w)\neq 0$.

First, consider case $(a)$. We notice that the trajectory $(V(t;w,v_R),W(t;w,v_R))$ between the point $(v_R,w)$ and $(v_R, P_{v_R}(w))$  does not cross $w$-nullcline and thus can be seen as $V=V(W;w,v_R)$ where $V(W;w,v_R)$ is the solution of
\begin{equation}\label{star1}
\frac{\mathrm{d}V}{\mathrm{d}W}=\frac{F(V)-W+I}{\eps(bV-W)}
\end{equation}
with the initial condition $V(w)=v_R$. Hence we get the following implicit equation for $P_{v_R}(w)$:
\begin{equation}\label{implicit1}
V(P_{v_R}(w);w,v_R)=v_{R} + \int_{w}^{P_{v_R}(w)}\frac{F(V(W;w,v_R))-W+I}{\eps(bV(W;w,v_R))-W)}\;dW=v_R. 
\end{equation}
We define a function 
\[
G(z,w):=\int_{w}^{z}\frac{F(V(W;w,v_R))-W+I}{\eps(bV(W;w,v_R))-W)}\;dW,
\]
where $V(W;w,v_R)$ is the solution of (\ref{star1}) with the initial condition  $V(w)=v_R$,  such that (\ref{implicit1}) is equivalent to $G(P_{v_R},w)=0$. 
Since the point $(v_R,P_{v_R}(w))$ lies apart from both the nullclines we compute
\[
\frac{\partial G}{\partial z}(P_{v_R}(w),w)=\frac{F(V(P_{v_R}(w),w,v_R))-P_{v_R}(w)+I}{\eps(bV(P_{v_R}(w);w,v_R)-P_{v_R}(w))}=\frac{F(v_R)-P_{v_R}+I}{\eps(bv_R-P_{v_R}(w))}\neq 0
\]
and by Implicit Function Theorem  we obtain that the mapping $\tilde{w}\mapsto P_{v_R}(\tilde{w})$ is a $C^1$-function in the neighbourhood of $\tilde{w}=w$. Consequently, $P_{v_R}^{\prime}(w)$ exists.  Moreover, by differentiating  the equation (\ref{implicit1}) with respect to $w$ we obtain that $P_{v_R}^{\prime}(w)$ satisfies 
\begin{equation}\label{star2}
\frac{\partial V}{\partial W}(P_{v_R}(w);w,v_R) P_{v_R}^{\prime}(w)+\frac{\partial V}{\partial w}(P_{v_R}(w);w,v_R)=0,
\end{equation}
where we abuse notation by letting $\frac{\partial V}{\partial w}$ denote the partial derivative of $V$ specifically with respect to the initial condition $w$, namely the second argument of $V$.

As the point $(v_R,P_{v_R}(w))$ lies apart both the nullclines, it follows that
\[
\frac{\partial V}{\partial W}(P_{v_R}(w);w,v_R)\neq 0.
\]
Hence, based on equation (\ref{star2}), it suffices to show that $\frac{\partial V}{\partial w}(P_{v_R}(w);w,v_R) \neq 0$.
By an application of Peano's Theorem (see equation (3.4) of \cite{phartman}),  
\[
\frac{\partial V}{\partial w}(P_{v_R}(w);w,v_R)=-\frac{\partial V}{\partial v_R}(P_{v_R}(w),w,v_R) \frac{F(v_R)-w+I}{a(bv_R-w)}, 
\]
with (see Corollary 3.1 of \cite{phartman}) 
\[
\frac{\partial V}{\partial v_R}(P_{v_R}(w);w,v_R)=\exp\left(\int_{w}^{P_{v_R}(w)}\frac{F^{\prime}(V)(bV-W)-b(F(V)-W+I)}{{\eps}(bV-W)^2}\;dW\right),
\]
 where $V=V(W;w,v_R)$.  Therefore, $\frac{\partial V}{\partial w}(P_{v_R}(w);w,v_R)\neq 0$, as desired.
 The proof for case (a) is completed.

 In case (b) one needs to notice that the trajectory $(V(t;w,v_R),W(t;w,v_R))$ between the point $(v_R,w)$ and $(v_R, P_{v_R}(w))$  crosses first the $v$-nullcline and then the $w$-nullcline. Therefore we cannot argue exactly as in case (a). Nevertheless, (b) can be reduced to (a) in the following way: There exists a neighbourhood $U$ of $(v_R,w)$ and the reset value $\hat{v}_R<v_R$ such that for every $(v_R,\tilde{w})\in U$ the trajectory $(V(t;\tilde{w},v_R),W(t;\tilde{w},v_R))$  crosses the line $v=\hat{v}_R$ exactly two times, say at points $P^1_{\hat{v}_R}(\tilde{w})$ and $P^2_{\hat{v}_R}(P^1_{\hat{v}_R}(\tilde{w}))$ such that $P^1_{\hat{v}_R}(\tilde{w})>F(\hat{v}_R)+I$ and $b\hat{v}_R<P^2_{\hat{v}_R}(P^1_{\hat{v}_R}(\tilde{w}))<F(\hat{v}_R)+I$ (i.e. $P^2_{\hat{v}_R}$ is such as in case (a): the crossing occurs between the two nullclines, not below both). Now we have $\Phi_{v_R}(w)=\Phi_{\hat{v}_R}(P^2_{\hat{v}_R}(P^1_{\hat{v}_R}(w)))$ and
 \[
 \Phi^{\prime}_{v_R}(w)=\Phi^{\prime}_{\hat{v}_R}(P^2_{\hat{v}_R}(P^1_{\hat{v}_R}(w))) (P^2_{\hat{v}_R})^{\prime}(P^1_{\hat{v}_R}(w)) (P^1_{\hat{v}_R})^{\prime}(w).
 \]
 Since $P^2_{\hat{v}_R}(P^1_{\hat{v}_R}(\tilde{w}))>b\hat{v}_R$ from (a) we have $(P^2_{\hat{v}_R})^{\prime}(P^1_{\hat{v}_R}(w))\neq 0$ and $\Phi^{\prime}_{\hat{v}_R}(P^2_{\hat{v}_R}(P^1_{\hat{v}_R}(w)))\neq 0$. By expressing the trajectories $(V,W)$ as the function $W=W(V)$, we similarly obtain 
 \[
 (P^1_{\hat{v}_R})^{\prime}(w)=\exp\left(\int_{v_R}^{\hat{v}_R}\frac{\eps(bV-F(V)-I)}{(F(V)-W(V;v_R,w)+I)^2}\;dV\right)\neq 0
 \]
 Hence  $\Phi^{\prime}_{v_R}(w)\neq 0$ also in (b)
\end{proof}

\begin{theorem}\label{non-degeneracy} Given $\eps>0$, if $v_R$ is sufficiently large such that $F'(v_R)>\eps$, then 
the point $w^*$ is non-degenerate; that is,  $\Phi_{v_R}^{\prime\prime}(w^*)<0$.
\end{theorem}
\begin{proof}
As in equation \eqref{eq:Phiprime}, for any $v>v_R$, each trajectory $W(v;v_R,w)$ with initial condition $(v_R,w), w<w^*$, satisfies 
\[
\frac{\partial W}{\partial w}(v; v_R,w) = \exp \left(-\int\limits_{v_R}^{v} \frac{\eps(F(s)+I-bs)}{(F(s)+I-W(s;v_R,w))^2} ds\right)
\]
 and
\[
\lim\limits_{w \to (w^*)^-} \frac{\partial W}{\partial w}(v; v_R,w) = 0.
\]
We already know  that $\Phi''(w)<0$ for $w<w^*$. Using the flow-box theorem, it is also obvious that $\Phi''(w)<0$ for $w > w^*$ close enough to $w^*$. We aim at proving that $\Phi''(w^*)<0$ by considering the limit of the difference quotient of $\Phi'$ on the left of $w^*$. Hence, since we already know that $\frac{\partial W}{\partial w}(v; v_R,w^*) = 0$ we want to prove that
\[
\lim\limits_{w \to (w^*)^-} \lim\limits_{v \to +\infty} \frac{\frac{\partial W}{\partial w}(v; v_R,w)}{w-w^*} < 0.
\]

By assumption, $F'(v_R)>\eps$; that is,  the reset line is bounded away from the knee of the $v$-nullcline. We introduce a parameter $\delta >0$ small but fixed and, for studying the $w$-limit, we only consider the values $w \in [w_{\min}, w^*]$ for some $w_{\min}<w^*$ and
\[
W(v_R+\delta ; v_R,w_{\min})>b(v_R+\delta).
\]
Hence, any trajectory starting from $(v_R,w)$ with $w \in [w_{\min}, w^*]$ remains above the $w$-nullcline at least for $v \in [v_R, v_R + \delta]$ and $W(v ; v_R,w)$ decreases over this  interval of $v$ values. Then, for $v>v_R+\delta$, we split the integral and exponentials as follows:
\begin{multline}
\frac{\frac{\partial W}{\partial w}(v; v_R,w)}{w-w^*} = \left(\frac{{\rm e}^{\left(-\int\limits_{v_R}^{v_R+\delta} \frac{\eps(F(s)+I-bs)}{(F(s)+I-W(s;v_R,w))^2} ds\right)}}{w-w^*}\right) {\rm e}^{\left(-\int\limits_{v_R+\delta}^{v} \frac{\eps(F(s)+I-bs)}{(F(s)+I-W(s;v_R,w))^2} ds\right)}
\end{multline}
First note that the second factor is well-defined even around $w=w^*$ and $v$ tending to $+\infty$ since
\[
W(v_R+\delta;v_R,w^*) < w^* = F(v_R)+I < F(v_R + \delta) +I.
\]
Hence, this factor converges uniformly in $w$ (in the vicinity of $w^*$) and $v$ (on $[v_R + \delta , +\infty[$) towards $(\Phi_{v_R+\delta})'(W(v_R+\delta;v_R,w^*))>0$ where $\Phi_{v_R+\delta}$ is the adaptation map with $v_R+\delta$ as reset value. On the other hand, the first term does not depend on $v$ and its limit as $w \to (w^*)^-$ is well-defined. 

Now, we focus on proving that this latter limit is strictly negative. 
We have, for $w \in [w_{\min},w^*),$
\begin{multline} \label{Integr}
J(w):=\int\limits_{v_R}^{v_R+\delta} \frac{\eps(F(s)+I-bs)}{(F(s)+I-W(s;v_R,w))^2} ds  \\
= \int\limits_{v_R}^{v_R+\delta} \frac{\eps(F(s)+I-bs)}{(F'(s)-\frac{\partial W}{\partial s}(s; v_R,w))(F(s)+I-W(s;v_R,w))}\frac{F'(s)-\frac{\partial W}{\partial s}(s; v_R,w)}{F(s)+I-W(s;v_R,w)} ds.
\end{multline}
The first factor in the integrand is strictly positive for $(s,w) \in [v_R, v_R+\delta] \times [w_{\min}, w^*]$. Moreover, from (\ref{eq:SubthreshDyn}) we have
\begin{multline}
 (F'(s)-\frac{\partial W}{\partial s}(s; v_R,w))(F(s)+I-W(s;v_R,w)) \\
  = F'(s)(F(s)+I-W(s;v_R,w)) - \frac{\partial W}{\partial s}(s; v_R,w)(F(s)+I-W(s;v_R,w)) \\
  = F'(s)(F(s)+I-W(s;v_R,w)) + \eps (W(s;v_R,w)-bs).
\end{multline}
Using this expression, one obtains
\begin{multline}
 \frac{\eps(F(s)+I-bs)}{(F'(s)-\frac{\partial W}{\partial s}(s; v_R,w))(F(s)+I-W(s;v_R,w))} \\
 	= \frac{\eps(F(s)+I-bs)}{\eps(F(s)+I-bs) + (F'(s)-\eps)(F(s)+I-W(s;v_R,w))} \leq 1
\end{multline}
since we assume $F'>\eps$ for $v>v_R$. It is worth noting that the constant $1$ is optimal since $W(v_R;v_R,w^*)=F(v_R)+I$. It follows from \eqref{Integr}
\begin{multline}
J(w) \leq \int\limits_{v_R}^{v_R+\delta} \frac{F'(s)-\frac{\partial W}{\partial s}(s; v_R,w)}{F(s)+I-W(s;v_R,w)} ds \\
\leq -\log \frac{w^*-w}{F(v_R + \delta) +I - W(v_R+\delta;v_R,w)} \leq - \log (K(w^*-w)).
\end{multline}
with
\[
\frac{1}{K}=F(v_R + \delta) +I - W(v_R+\delta;v_R,w_{\min}) >0.
\]
Finally, it follows that for any $w \in [w_{\min},w^*),$
\[
{\rm e}^{-J(w)} \geq K(w^*-w)
\]
And since, $w-w^*<0$,
\[
\frac{{\rm e}^{-J(w)}}{w-w^*} \leq -K < 0, 
\]
such that the limit for $w \to (w^*)^-$ is strictly negative, which completes the proof.
\end{proof}

\begin{remark} Note that Theorems \ref{uniq-crit} and \ref{non-degeneracy} do not require $\eps$ small, therefore they apply to the adaptation map in general, not only near the singularly perturbed limit.
\end{remark}

\subsection{Topological chaos}\label{TopolChaos}
Probably the most common definition of chaos is the one due to Devaney \cite{devaney}, which states that a continuous map $f:X \to X$ on a compact metric space $X$ is chaotic if there exists a compact invariant subset $Y\subset X$ (called a $D$-chaotic set ) such that $f\vert_{Y}$ is transitive, the set of  periodic points  of $f\vert_{Y}$  is dense in $Y$ and $f\vert_{Y}$ has a sensitive dependence on initial conditions\footnote{The original definition due to Devaney takes $Y=X$, i.e. all the three conditions must hold on the whole domain $X$. However, usually the more general situation  where $Y\subset X$ is considered (see e.g. \cite{aulbach} and references therein). We also take this more general approach.} (where  $f\vert_{Y}$ denotes the restriction of $f$ to the set $Y$).  A map satisfying these properties is called Devaney- or $D$-chaotic.  We refer the reader e.g. to \cite{aulbach,banks,guckenheimer, silverman} for definitions of transitivity and sensitive dependence on initial conditions as well as results on redundancy of the last one or even sometimes the last two conditions in the definition of $D$-chaos. 

The other notions of topological chaos include, among others, positive topological entropy, chaos in the sense of Block and Coppel (see e.g. \cite{aulbach,blockcoppel}), Li-Yorke chaos (weaker than D-chaos and Block-Coppel chaos, see \cite{aulbach}), or even weaker notions such as the sensitive dependence on initial conditions itself or the existence of a period-3 orbit. It is not our aim to discuss here all these notions but just to make the observation that for the adaptation map all of them are very likely to occur. 

To express the complexity of the dynamics, it is also useful to look for ``horseshoes'', defined for a one-dimensional map as follows.
\begin{definition}\label{Horseshoe}  A continuous map $f:I\to I$ (where $I\subset \mathbb{R}$ is a compact interval) has a $2$-horseshoe (or in other words, $f$ is \emph{turbulent}), if there exist two closed-subintervals of $I$, $A_1$ and $A_2$, with disjoint interiors, such that
\begin{displaymath}
(A_1\cup A_2)\subseteq (f(A_1)\cap f(A_2)).
\end{displaymath}
\end{definition}
While $\Phi$ itself cannot feature a horseshoe, its iterates $\Phi^k$ for $k\geq 2$ may, resulting in chaotic dynamics of the map.

To carry out the proof of the forthcoming Theorem \ref{NonalternatingChaos} that establishes the chaotic nature of the adaptation map, we need to introduce one more definition:
\begin{definition} [see e.g. \cite{blockcoppel}] A trajectory $\{\Phi^n(w)\}_{n\geq 0}$ of some point $w\in \mathbb{R}$ will be called \emph{alternating} if $\Phi^k(w)<\Phi^j(w)$ for all even integers $k$ and all odd $j$, or $\Phi^k(w)>\Phi^j(w)$ for all even integers $k$ and all odd $j$.
\end{definition}

\begin{theorem}\label{NonalternatingChaos} Suppose that $\Phi^2(w^*)<\Phi^3(w^*)<w^*<\Phi(w^*)$.  Then
\begin{enumerate}
\item the map $\Phi$ has periodic orbits of all periods
\item $\Phi^m$ is turbulent for some $m\in \mathbb{N}$
\item $\Phi$ has positive topological entropy
\item$\Phi$ is chaotic in the sense of Li-Yorke, Block and Coppel and Devaney (with some $D$-chaotic set $Y\subset [\Phi^2(w^*),\Phi(w^*)]$).
\end{enumerate}
\end{theorem}

\begin{proof} The proof relies on the known results of one-dimensional dynamics. In particular, the first statement is a consequence of  Theorem II.9 in \cite{blockcoppel}, which ensures that the assumption $\Phi^3(w^*)<w^*<\Phi(w^*)$ implies that $\Phi$ has an orbit of period $3$, and consequently periodic orbits of all periods by Sharkovskii's theorem (see e.g.~\cite{devaney}). 

The second statement relies on Theorem II.12 of \cite{blockcoppel} which implies, under our condition  $\Phi^2(w^*)<w^*<\Phi(w^*)$, that if $\Phi^2$ was not turbulent, then the orbit of $w^*$ would be  necessarily alternating. However, this is not the case, since the assumption $\Phi^2(w^*)<\Phi^3(w^*)<w^*$ implies that $\Phi^3(w^*)<\Phi^4(w^*)$ (because $\Phi$ is strictly increasing on $(-\infty, w^*)$), and thus 
the orbit of $w^*$ under $\Phi$ is non-alternating. Therefore $\Phi^2$ is necessarily turbulent, proving statement 2.

Statements 3 and 4 are general implications of 1 and 2. Indeed, statement 3 follows from 1 since the existence of a periodic point whose period is not a power of two is in this case equivalent to positive topological entropy (see Theorem 3.22 in \cite{colleteckmann2} and Corollary 3 in \cite{MMhorseshoe}). Statement 4 can be derived as a general consequence of 2 based on Proposition 3.3, Theorem 4.1 and Theorem 4.2 in \cite{aulbach}. \end{proof}

\begin{remark}
Theorem~\ref{NonalternatingChaos} could read ``topological chaos occurs almost all the time'', since the condition  $\Phi^2(w^*)<\Phi^3(w^*)<w^*<\Phi(w^*)$ is satisfied, for example, for any choice of $v_R$ such that $F(v_R)+I>p_0+2d$ and  any $\eps$ small enough (see Lemma \ref{ShapeLemma} and Remark following it).
Note also that a condition for existence of periodic orbits of all periods was given in~\cite[Theorem 3.4]{touboul-brette:09}. The assumption of Theorem~\ref{NonalternatingChaos} is weaker than the previous result and thus covers more cases. 
\end{remark}

While the above results prove the existence of an infinite\footnote{ $Y$ is necessarily infinite because the sensitive dependence on initial conditions holds on this set.} set $Y$ on which the map is chaotic in the sense of the above Theorem \ref{NonalternatingChaos}, they do not ensure that the chaotic behavior is generic for arbitrary trajectories. In particular, the set $Y$ may be small, even of zero Lebesgue measure. Similarly, the existing period-3 orbit might be stable and attracting for almost all initial conditions. The following section focuses on a stronger notion of chaos, namely \emph{metric chaos}. 

\subsection{Metric chaos} \label{ChaosUnimod}
In this section we discuss the theoretical justification for the emergence of visible (metric) chaos occurring in our bifurcation diagrams (see for example the $\eps = 0.4$ panel of Figure  \ref{fig:different_epsilons}). For this purpose, we come back to the theory of unimodal maps (see e.g.~\cite{thunberg}) and show that the standard theory may, under technical conditions on the adaptation map, directly apply to our system and justify the existence of metric chaos at the transitions. We recall the following definitions:

\begin{definition}[Metric chaos]\label{metricchaos} We say that $\Phi$ is \emph{chaotic} if it admits an absolutely continuous invariant probability measure (\emph{acip}) $\mu$, i.e. an invariant measure that is finite, normalized and has density with respect to the Lebesgue measure.
\end{definition}

\begin{definition}\label{misiurewiczmap}
A map $\Phi$ is called a \emph{Misiurewicz map} if it has no periodic attractors and if critical orbits (i.e. forward orbit of the critical points) do not accumulate on critical points, that is, if
\begin{equation}\label{generalMMcondition}
\mathcal{C} \cap \omega(\mathcal{C}) =\emptyset,
\end{equation}
where $\mathcal{C}$ denotes the set of critical points of $\Phi$ and  $\omega(\mathcal{C})$ is its $\omega$-limit set.
\end{definition}

For some bounded interval of parameter values $v_R\in [v_{R_1},v_{R_2}]=:\mathcal{V}$ consider the one-parameter family of adaptation maps $\{\Phi_{v_R}\}_{v_R\in \mathcal{V}}$. We assume, as previously, that for each $v_R\in\mathcal{V}$ we have $\Phi^2_{v_R}(w^*)<w^*<\Phi_{v_R}(w^*)$. Under this condition it is not hard to construct an associated family of unimodal maps $\tilde{\Phi}_{{v_{R}}}$ whose orbits are in a one-to-one correspondence with those of $\Phi_{v_{R}}$ (at least after a transient period). To this end, we start with a change of coordinates that  makes the singular point $w^{*}$ independent of parameter $v_{R}$, and define:
\[\bar{\Phi}_{v_R}(w):=h^{-1}(\Phi_{v_R}(h(w))) \textrm{ with }  h:w\mapsto w-w^*_{v_R}=w-F(v_R)-I.\] 
These maps have their critical point at $\tilde{w}^*=0$ for all $v_{R}$. Let $J=[A,B]$ be a closed interval strictly containing all dynamical cores $[\bar{\Phi}^2_{v_R}(w^*),\bar{\Phi}_{v_R}(w^*)]$ of the family $\bar{\Phi}_{v_R}$. We can define a family of unimodal maps $\tilde{\Phi}_{{v_{R}}}$ on $J$ that are equal to corresponding maps $\bar{\Phi}_{v_R}$ on $[\bar{\Phi}^2_{v_R}(w^*),\bar{\Phi}_{v_R}(w^*)]$, with $A$ being a repelling fixed point and $\tilde{\Phi}_{v_R}(B)=A$ and with no fixed points in $[A,\bar{\Phi}^2_{v_R}(w^*)]$. After a finite number of iterates, the orbits of $\tilde{\Phi}_{{v_{R}}}$ are identical to those of $\bar{\Phi}_{v_{R}}$, which are simple translations of those of $\Phi_{v_{R}}$. Therefore, the asymptotic dynamics of $\tilde{\Phi}_{{v_{R}}}$ and the original map $\Phi_{{v_{R}}}$ are the same. By studying the  unimodal maps $\tilde{\Phi}_{{v_{R}}}$, we can bring the well-developed theory of unimodal maps~\cite{deMeloStrien_book,thunberg} to bear to characterize the non-transient properties of the orbits of $\Phi_{{v_{R}}}$. 

It is clear that one can perform the above construction of the family $\{\tilde{\Phi}_{{v_{R}}}\}$ such that the following standard conditions are satisfied:
\begin{description}
\item[[I]] $\tilde{\Phi}_{v_R}$ is a one-parameter family of $C^3$ unimodal maps of an interval $J$. 
\item[[II]] Each $\tilde{\Phi}_{v_R}$ has a unique and nondegenerate critical point $w^*$ (independent of
$v_R$).
\item[[III]] Each $\tilde{\Phi}_{v_R}$ has a repelling fixed point on the boundary of $J$.
\item[[IV]] The map $(w,v_R) \mapsto (\tilde{\Phi}_{v_R}(w), \tilde{\Phi}_{v_R}^{\prime}(w), \tilde{\Phi}_{v_R}^{\prime\prime}(w))$ is $C^1$.
\end{description}

Such maps display metric chaos as soon as:
\begin{description}
	\item[[V]] there exists $\bar{v}_{R}\in \mathcal{V}$ such that $\Phi_{\bar{v}_R}$ is a Misiurewicz map.
\end{description}
Unfortunately, it is complex to establish condition \textbf{[V]}. 
The main difficulty in verifying \textbf{[V]} is in assuring that there are no periodic attractors for the map $\Phi_{\bar{v}_R}$.  Showing the absence of attractive periodic orbits is a complex task for our general model, and even simpler sufficient conditions such as those relying on negative Schwarzian derivatives at ${{\bar{v}_{R}}}$ present deep difficulties to demonstrate.
Another issue is that the  accumulation of the orbit of $w^*$ in a neighbourhood of $w^*$ must not occur; this can be avoided in particular when the critical point $w^*$ is mapped in a few iterations onto the unstable fixed point $w^{f}_{{\bar{v}_{R}}}$ (see~\cite[Section 6.2]{thunberg}).  This condition is actually satisfied for some intermediate parameter value $\bar{v}_{R}$ in the period-incrementing transition.  Indeed, it is relatively easy to show, using the intermediate value theorem, that:
\begin{proposition} For sufficiently small $\eps$, the  family $\{\Phi_{v_R}\}_{v_R \in \mathcal{V}}$ of adaptation maps undergoes period-incrementing transitions such that between any two intervals ${J}_k=[a_k,b_k]$ and ${J}_{k+1}=[a_{k+1},b_{k+1}]$  of $v_R$ values, corresponding, respectively, to admissible $k$ and $k+1$ attracting periodic orbits, there exists a parameter value $\bar{v}_R\in (b_k,a_{k+1})$ such that 
\[
(\Phi_{\bar{v}_R})^{k+1}(w^*_{\bar{v}_R})=w^f_{\bar{v}_{R}},  
\]
i.e. the critical point is mapped into a few steps onto the fixed point. 
\end{proposition}
The correspondence of the above result is obviously also satisfied for the family $\{\tilde{\Phi}_{v_R}\}$. Moreover, sufficient conditions for the fixed point $w^f_{\bar{v}_{R}}$ to be unstable are provided in  Lemma~\ref{ShapeLemma}.

For completeness, one would also require some additional non-degeneracy condition on how the unstable fixed point and the point that is eventually mapped onto it evolve with the change of the parameter $v_R$  in the neighbourhood of $\bar{v}_{R}$. This condition is technical and therefore for its precise statement we refer to \cite{thunberg} (see condition H6 therein). It is difficult to demonstrate rigorously at full generality, yet it is completely generic.   

Our analysis and extensive numerical simulations suggest that these technical conditions are generally met. When they hold, the theory of unimodal maps~\cite[Theorem 18 and Corollary 19]{thunberg} yields the conclusion that there exist constants $\gamma > 0$ and $C > 0$ and a positive measure set $E\subset \mathcal{V}$ with $\bar{v}_R\in E$ as a Lebesgue density point, such that
\[
\forall v_R \in E, \quad \forall n \in \N^*, \quad
\vert (\Phi^n_{v_R})^{\prime}(\Phi_{v_R}(w^*))\vert \geq C\mathrm{e}^{\gamma n}.
\]
Furthermore, under the above conditions and if the map has a negative Schwarzian derivative in a neighborhood $\bar{\mathcal{V}}\subset\mathcal{V}$ of $v_{R}$, then $E\subset\bar{\mathcal{V}}$ can be chosen so that for all $v_R \in E$ $\Phi_{v_R}$ exhibits metric chaos  with an \emph{acip} $\mu_{v_R}$, describing the asymptotic behavior of almost all orbits, and that has a positive Lyapunov exponent almost everywhere:
\begin{equation}\label{eq:lyapunexp}
\lim_{n\to\infty}\frac{1}{n}\log \vert (\Phi_{v_R}^n)^{\prime}(w)\vert = \kappa>0 \quad \textrm{for a.a.} \ w\in \mathbb{R}
\end{equation}

We illustrate the application of this theory in Figure~\ref{fig:ACIP} within the range of values associated with the transition from bursts with 4 to bursts with 5 spikes. We evaluated the critical value $\bar{v}_{R}$ numerically and simulated the orbits of the system for distinct initial conditions. These orbits collapse on the unique measure depicted in the figure.

\begin{figure}
\centering
\includegraphics[width=\textwidth]{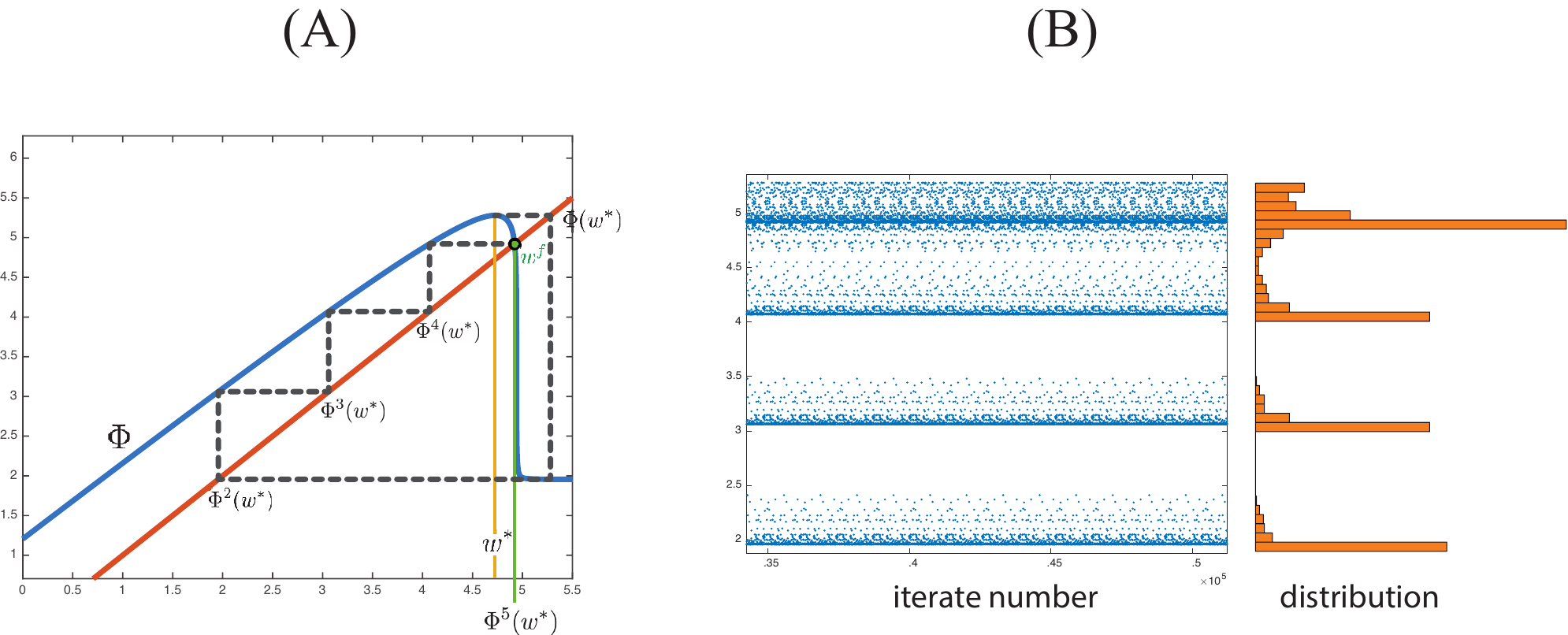}
\caption{Metric chaos in the quartic integrate-and-fire neuron with standard parameters~\eqref{eq:StandardParams} and $\eps=0.4$. (A) Adaptation map in the case where $\Phi^{5}(w^{*})=w^{f}$, a fixed point, found by fine-tuning the value of the reset voltage (here, $v_{R}=1.2226$). The sequence of iterates is chaotic (strictly positive Lyapunov exponent~\eqref{eq:lyapunexp} with for instance $\kappa>6$ for all $w$ tested). (B) Iterates of $\Phi$ in the same setting yield a complex chaotic sequence. The $15\,000$ final iterates are depicted, together with the obtained distribution of points. }
\label{fig:ACIP}
\end{figure}

\section{Discussion} 
  The study of the bursting patterns in nonlinear adaptive integrate-and-fire neuron models led us to identify and elucidate the  underlying period-incrementing structure that organizes spike patterns in a regular progression as parameters (particularly, the reset value of the membrane potential) are varied. The structure exhibited here is general and does not depend on the specific adaptive integrate-and-fire model considered. We have shown that this structure is present in particular when the adaptation variable is much slower than the voltage variable. Indeed, we have established that in the limit of perfect separation of timescales, the sequence of adaptation values approach orbits of a discontinuous piecewise linear discrete dynamical system having, for any set of parameters, a unique globally attractive periodic orbit, with a  period that is incremented instantaneously as parameters are varied. In the full system, we have shown that while the period-incrementing structure was globally conserved, transitions are no longer instantaneous, and the periodic orbits bifurcate and lose stability. We have investigated in more detail the presence of chaos during these transitions and shown that within these regions, the system exhibits topological chaos and may display metric chaos as well, under the additional assumption that the adaptation map has a negative Schwarzian derivative, at least for a specific value of the voltage reset. 

From the mathematical viewpoint, the relative simplicity of the model allowed us to make significant progress in  characterizing the complex period-incrementing transition structure. In the class of models investigated, however, the fact that the adaptation map is not known in closed form raises some difficulties and in particular imposes limitations to the results shown. With the same methodology, one may thus achieve stronger results for a specific choice of model in which the adaptation map satisfies a few additional properties. In particular, establishing that a particular map features  a negative Schwarzian derivative away from the critical point would allow an extension of the results on metric chaos developed  in Section~\ref{ChaosUnimod}  as well as a more complete characterization of the structures of the topological and metric attractors and their inter-relationships (see~\cite{milnor85,blokh-lyubich,BKNvS} as well as the survey article~\cite{thunberg} and references therein). For instance, considering our result on the non-degeneracy of $w^*$, we know that the Schwarzian derivative is negative in a neighborhood of the critical point, which allows the application of  van Strien's theory~\cite{Strien} showing the uniform boundedness of the periods of periodic attractors and non-hyperbolic periodic orbits of $\Phi_{\bar{v}_R}$. Using the same theory, if all periodic points of $\Phi_{\bar{v}_R}$ are hyperbolic and repelling, then $\Phi_{\bar{v}_R}$ admits an \emph{acip} of positive entropy.

A number of further questions arise at this stage. For instance, since the adaptation map $\Phi$ is continuous and unimodal, a natural question is to investigate to what extent the dynamics of $\Phi$ may be similar to that of the canonical logistic map $F_{\mu}: x\in [0,1] \mapsto \mu x(1 - x)$ with $\mu\in (0,4]$. While Milnor-Thurston kneading theory~\cite{milnor-thurston,deMeloStrien_book} ensures that our unimodal extension of the adaptation map is \emph{semiconjugated} to a map $\{F_{\mu}\}$ through a continuous surjective and non-decreasing map, the conclusions that may be drawn remain limited. In particular, one would need again to show the uniqueness of periodic attractors to ensure that the semi-conjugacy is actually a conjugacy~\cite{deMeloStrien2}, but typically this is not the case and the dynamics of $\Phi$ is not completely equivalent to that of the logistic map. 

From the application viewpoint, these results yield a deeper understanding of the dynamics and its parameter-dependence for a widely used class of hybrid neuronal models, in the case where the subthreshold dynamics has no fixed point.  This configuration corresponds to settings where the input to the neuron is sufficiently large.  Interestingly, one scenario where neurons receive unusually high input is during seizures within the epileptic brain (e.g. \cite{meijer} and references therein).  Chaotic dynamics has been associated with epileptic brain dynamics ~\cite{schiff94}; it is also possible that seizures represent transitions from chaotic to more regular dynamics within high input regimes~\cite{iasemidis96}.  Hence, our work has possible relevance for the use of bidimensional hybrid neuronal models to study epileptic dynamics.  Our analysis is based on an adaptation map, which can be defined on the whole real line in this situation.  In the companion paper~\cite{paper2}, we will switch gears and investigate the case where the subthreshold system has two unstable fixed points, a spiral and a saddle, with a heteroclinic orbit from the former to the latter. In that setting, we obtain and study the corresponding, distinctive forms of dynamics that arise, which  feature alternations of small oscillations and spikes or bursts, and are known as mixed-mode oscillations.  Again, our analysis will  rely heavily  on the relatively simple geometric structure of the hybrid system. 

\noindent {\bf Acknowledgements:}
J. Rubin was partly supported by US National Science Foundation awards DMS 1312508 and 1612913.
J. Signerska-Rynkowska was supported by Polish National Science Centre grant 2014/15/B/ST1/01710. 

\bigskip


\begin{thebibliography}{10}

\bibitem{aulbach}
{\sc B.~Aulbach and B.~Kieninger}, {\em On three definitions of chaos},
  Nonlinear Dyn. Syst. Theory, 1 (2001), pp.~23--37.

\bibitem{avrutin}
{\sc V.~Avrutin, A.~Granados, and M.~Schanz}, {\em Sufficient conditions for a
  period incrementing big bang bifurcation in one-dimensional maps},
  Nonlinearity, 24 (2011), p.~2575?2598.

\bibitem{banks}
{\sc J.~Banks, J.~Brooks, G.~Cairns, G.~Davis, and P.~Stacey}, {\em On
  {{D}}evaney's definition of chaos}, The American Mathematical Monthly, 99
  (1992), pp.~332--334.

\bibitem{belykh2005synchronization}
{\sc I.~Belykh, E.~de~Lange, and M.~Hasler}, {\em Synchronization of
  bursting neurons: What matters in the network topology}, Physical review
  letters, 94 (2005), p.~188101.

\bibitem{blockcoppel}
{\sc L.S. Block and W.A. Coppel}, {\em Dynamics in One Dimension},
  Springer-Verlag, 1992.

\bibitem{blokh-lyubich}
{\sc A.~M. Blokh and M.~Yu. Lyubich}, {\em Measurable dynamics of s-unimodal
  maps of the interval}, Ann. Sci. \'{E}cole Norm. Sup. (4), 24 (1991),
  pp.~545--573.

\bibitem{brette-gerstner:05}
{\sc R.~Brette and W.~Gerstner}, {\em Adaptive exponential integrate-and-fire
  model as an effective description of neuronal activity}, Journal of
  Neurophysiology, 94 (2005), pp.~3637--3642.

\bibitem{BKNvS}
{\sc H.~Bruin, G.~Keller, T.~Nowicki, and S.~van Strien}, {\em Wild cantor
  attractors exist.}, Ann. of Math., 143 (1996), pp.~97--130.

\bibitem{brunel2007lapicque}
{\sc N.~Brunel and M.~Van~Rossum}, {\em Lapicque's 1907 paper: from
  frogs to integrate-and-fire}, Biological cybernetics, 97 (2007),
  pp.~337--339.

\bibitem{colleteckmann2}
{\sc P.~Collet and J-P. Eckmann}, {\em Concepts and results in chaotic
  dynamics: a short course.}, Theoretical and Mathematical Physics,
  Springer-Verlag, Berlin, 2006.

\bibitem{burstbook}
{\sc S.~Coombes and P.~C.~Bressloff}, {\em Bursting: the genesis of
  rhythm in the nervous system}, World Scientific, 2005.

\bibitem{deM}
{\sc P.~de~{M}aesschalck and F.~Dumortier}, {\em Time analysis and
  entry--exit relation near planar turning points}, Journal of Differential
  Equations, 215 (2005), pp.~225--267.

\bibitem{deMeloStrien2}
{\sc W.~de~Melo and S.~van Strien}, {\em One-dimensional dynamics: the
  schwarzian derivative and beyond}, Bull. Amer. Math. Soc. (N.S.), 18 (1988),
  pp.~159--162.

\bibitem{deMeloStrien_book}
\leavevmode\vrule height 2pt depth -1.6pt width 23pt, {\em One-dimensional
  dynamics.}, Results in Mathematics and Related Areas (3), Springer-Verlag,
  Berlin, 1993.

\bibitem{Desroches13}
{\sc M.~Desroche, T.~J~Kaper, and M.~Krupa}, {\em Mixed-mode
  bursting oscillations: Dynamics created by a slow passage through
  spike-adding canard explosion in a square-wave burster}, Chaos, 23 (2013).

\bibitem{destexhe1998}
{\sc A.~Destexhe, D.~Contreras, and M.~Steriade}, {\em Mechanisms
  underlying the synchronizing action of corticothalamic feedback through
  inhibition of thalamic relay cells}, Journal of neurophysiology, 79 (1998),
  pp.~999--1016.

\bibitem{devaney}
{\sc R.L. Devaney}, {\em An Introduction to Chaotic Dynamical Systems},
  Westview Press, 2003.

\bibitem{dumortier1996}
{\sc F.~Dumortier and R.~H.~Roussarie}, {\em Canard cycles and center
  manifolds}, vol.~577, American Mathematical Soc., 1996.

\bibitem{foxall2012contraction}
{\sc E.~Foxall, R.~Edwards, S.~Ibrahim, and P.~van~den Driessche}, {\em A
  contraction argument for two-dimensional spiking neuron models}, SIAM Journal
  on Applied Dynamical Systems, 11 (2012), pp.~540--566.

\bibitem{guckenheimer}
{\sc J.~Guckenheimer}, {\em Sensitive dependence to initial conditions for
  one-dimensional maps.}, Comm. Math. Phys., 70 (1979), p.~133?160.

\bibitem{phartman}
{\sc P.~Hartman}, {\em Ordinary Differential Equations}, Classics in Applied
  Mathematics, 38, SIAM, 1982.
\newblock Corrected reprint of the second (1982) edition.

\bibitem{iasemidis96}
{\sc L.~D. Iasemidis and J~C. Sackellares}, {\em Review: Chaos theory
  and epilepsy}, The Neuroscientist, 2 (1996), pp.~118--126.

\bibitem{ibarz2011map}
{\sc B.~Ibarz, J.M.~Casado, and Miguel~AF Sanju{\'a}n}, {\em
  Map-based models in neuronal dynamics}, Physics Reports, 501 (2011),
  pp.~1--74.

\bibitem{izhikevich:00}
{\sc E.M. Izhikevich}, {\em Neural excitability, spiking, and bursting},
  International Journal of Bifurcation and Chaos, 10 (2000), pp.~1171--1266.

\bibitem{izhikevich:03}
\leavevmode\vrule height 2pt depth -1.6pt width 23pt, {\em Simple model of
  spiking neurons}, IEEE Transactions on Neural Networks, 14 (2003),
  pp.~1569--1572.

\bibitem{izhikevich:04}
\leavevmode\vrule height 2pt depth -1.6pt width 23pt, {\em Which model to use
  for cortical spiking neurons?}, IEEE Trans Neural Netw, 15 (2004),
  pp.~1063--1070.

\bibitem{izhikevich:07}
\leavevmode\vrule height 2pt depth -1.6pt width 23pt, {\em {Dynamical Systems
  in Neuroscience: The Geometry of Excitability And Bursting}}, MIT Press,
  2007.
  
\bibitem{izhikevich2006bursting}
\leavevmode\vrule height 2pt depth -1.6pt width 23pt, {\em Bursting}, Scholarpedia, 1 (2006), p.~1300.

\bibitem{izhikevich-edelman:08}
{\sc E.M. Izhikevich and G.~M. Edelman}, {\em {L}arge-scale model of mammalian
  thalamocortical systems.}, Proc Natl Acad Sci USA, 105 (2008),
  pp.~3593--3598.

\bibitem{izhikevich2003bursts}
{\sc E.~M Izhikevich, N.~S Desai, E.~C Walcott, and Frank~C
  Hoppensteadt}, {\em Bursts as a unit of neural information: selective
  communication via resonance}, Trends in neurosciences, 26 (2003),
  pp.~161--167.

\bibitem{jia:12}
{\sc B.~Jia, H.~Gu, L.~Li, and X.~Zhao}, {\em Dynamics of
  period-doubling bifurcation to chaos in the spontaneous neural firing
  patterns}, Cognitive neurodynamics, 6 (2012), pp.~89--106.

\bibitem{jimenez2013}
{\sc N.~D Jimenez, S.~Mihalas, R.~Brown, E.~Niebur, and
  J.~Rubin}, {\em Locally contractive dynamics in generalized
  integrate-and-fire neurons}, SIAM journal on applied dynamical systems, 12
  (2013), pp.~1474--1514.

\bibitem{jolivet-kobayashi-etal:08}
{\sc R.~Jolivet, R.~Kobayashi, A.~Rauch, R.~Naud, S.~Shinomoto, and
  W.~Gerstner}, {\em {A benchmark test for a quantitative assessment of simple
  neuron models}}, Journal of Neuroscience Methods, 169 (2008), pp.~417--424.

\bibitem{juan:10}
{\sc M.0~Juan, L.~Yu-Ye, W.~Chun-Ling, Y.~Ming-Hao, G.~Hua-Guang,
  Q.~Shi-Xian, and R.~Wei}, {\em Interpreting a period-adding bifurcation
  scenario in neural bursting patterns using border-collision bifurcation in a
  discontinuous map of a slow control variable}, Chinese Physics B, 19 (2010),
  p.~080513.

\bibitem{kepecs-lisman:03}
{\sc A~Kepecs and J~Lisman}, {\em Information encoding and computation with
  spikes and bursts}, Network: Computation in Neural Systems, 14 (2003),
  pp.~103--118.

\bibitem{kepecs-lisman:02}
{\sc A~Kepecs, X-J Wang, and J~Lisman}, {\em Bursting neurons signal input
  slope}, Journal of Neuroscience, 22 (2002), pp.~9053--62.

\bibitem{lapicque:07}
{\sc L.~Lapicque}, {\em Recherches quantitatifs sur l'excitation des nerfs
  traitee comme une polarisation}, J. Physiol. Paris, 9 (1907), pp.~620--635.

\bibitem{Terman99}
{\sc E.~Lee and D.~Terman}, {\em Uniqueness and stability of periodic
  bursting solutions}, Journal of Differential Equations, 158 (1999),
  pp.~48--78.

\bibitem{levi:90}
{\sc M.~Levi}, {\em A period-adding phenomenon}, SIAM Journal on Applied
  Mathematics, 50 (1990), pp.~943--955.

\bibitem{lindsey2012}
{\sc B.~G Lindsey, I.~A Rybak, and J.~C Smith}, {\em Computational
  models and emergent properties of respiratory neural networks}, Comprehensive
  Physiology,  (2012).

\bibitem{manica2010}
{\sc E.~Manica, G.~Medvedev, and J.~E Rubin}, {\em First
  return maps for the dynamics of synaptically coupled conditional bursters},
  Biological Cybernetics, 103 (2010), pp.~87--104.

\bibitem{marder:00}
{\sc E.~Marder}, {\em Motor pattern generation}, Current Opinion in
  Neurobiology, 10 (2000), pp.~691--698.

\bibitem{marder2001}
{\sc E.~Marder and D.~Bucher}, {\em Central pattern generators and the
  control of rhythmic movements}, Current Biology, 11 (2001), pp.~R986--R996.

\bibitem{medvedev2005reduction}
{\sc G.~S Medvedev}, {\em Reduction of a model of an excitable cell to a
  one-dimensional map}, Physica D: Nonlinear Phenomena, 202 (2005), pp.~37--59.

\bibitem{meijer}
{\sc H.~GE Meijer, T.~L Eissa, B.~Kiewiet, J.~F Neuman, C.~A
  Schevon, R.~G Emerson, R.~R Goodman, G.~M McKhann, C.~J
  Marcuccilli, A.~K Tryba, J.D.~Cowan, S.A.~van Gils, W.~van Drongelen}, {\em Modeling focal epileptic activity
  in the wilson--cowan model with depolarization block}, The Journal of
  Mathematical Neuroscience (JMN), 5 (2015), p.~1.

\bibitem{milnor85}
{\sc J.~Milnor}, {\em On the concept of attractor.}, Comm. Math. Phys., 99
  (1985), pp.~177--195.

\bibitem{milnor-thurston}
{\sc J.~Milnor and W. Thurston}, {\em On iterated maps of the interval.},
  Dynamical systems. Lecture Notes in Math.,, Springer, Berlin, 1988.

\bibitem{MMhorseshoe}
{\sc M.~Misiurewicz}, {\em Horseshoes for continuous mappings of the
  interval.}, in Dynamical Systems, C.~Marchioro, ed., C.I.M.E. Summer Schools,
  Springer, 2011, ch.~2, pp.~125--135.

\bibitem{naud-macille-etal:08}
{\sc R.~Naud, N.~Macille, C.~Clopath, and W.~Gerstner}, {\em Firing patterns in
  the adaptive exponential integrate-and-fire model}, Biological Cybernetics,
  99 (2008), pp.~335--347.

\bibitem{oswald2004}
{\sc A-M.~M Oswald, M.~J. Chacron, B. Doiron, J. Bastian, and
  L. Maler}, {\em Parallel processing of sensory input by bursts and
  isolated spikes}, The Journal of Neuroscience, 24 (2004), pp.~4351--4362.

\bibitem{pakdaman:95}
{\sc K. Pakdaman, J-F. Vibert, E. Boussard, and
  N. Azmy}, {\em Single neuron with recurrent excitation: Effect of the
  transmission delay}, Neural Network, 9 (1996), pp.~797--818.

\bibitem{prince:78}
{\sc D~A. Prince}, {\em Neurophysiology of epilepsy}, Annual Review of
  Neuroscience, 1 (1978), pp.~395--415.

\bibitem{pring-budd:10}
{\sc S.~R. Pring and C.~J. Budd}, {\em The dynamics of regularized
  discontinuous maps with applications to impacting systems}, SIAM {J}ournal on
  {A}pplied {D}ynamical {S}ystems, 9 (2010), pp.~188--219.

\bibitem{pring}
{\sc S.~R. Pring and C.~J. Budd}, {\em The dynamics of regularized
  discontinuous maps with applications to impacting systems.}, SIAM J. Appl.
  Dyn. Syst., 9 (2010), pp.~188--219.

\bibitem{rinzel}
{\sc J.~Rinzel}, {\em A formal classification of bursting mechanisms in
  excitable systems}, in Mathematical topics in population biology,
  morphogenesis and neurosciences, Springer, 1987, pp.~267--281.

\bibitem{rinzel-troy:83}
{\sc J.~Rinzel and W.~C. Troy}, {\em Lecture Notes in Biomathematics},
  Springer, Berlin, Garden City, N.Y., 1983, ch.~A one-variable map analysis of
  bursting in the Belousov-Zhabotinskii reaction.

\bibitem{paper2}
{\sc J.E. Rubin, J.~Signerska-Rynkowska, J.~Touboul, and A.~Vidal}, {\em Wild
  oscillations in a nonlinear neuron model with resets: (ii) {M}ixed-mode
  oscillations},  (submitted).

\bibitem{rubin2012}
{\sc J.E. Rubin, C.C. McIntyre, R.~S. Turner, and T.
  Wichmann}, {\em Basal ganglia activity patterns in parkinsonism and
  computational modeling of their downstream effects}, European Journal of
  Neuroscience, 36 (2012), pp.~2213--2228.

\bibitem{rulkov2004oscillations}
{\sc N.F.~Rulkov, I.~Timofeev, and M.~Bazhenov}, {\em Oscillations in
  large-scale cortical networks: map-based model}, Journal of computational
  neuroscience, 17 (2004), pp.~203--223.

\bibitem{rulkov2002modeling}
{\sc N.~F. Rulkov}, {\em Modeling of spiking-bursting neural behavior using
  two-dimensional map}, Physical Review E, 65 (2002), p.~041922.

\bibitem{samengo2010}
{\sc In{\'e}s Samengo and Marcelo~A Montemurro}, {\em Conversion of phase
  information into a spike-count code by bursting neurons}, P{L}o{S} One, 5
  (2010), p.~e9669.

\bibitem{schiff94}
{\sc S.~J.~Schiff, K.~Jerger, D.~H.~Duong, T.~Chang, M.~L.~Spano and
  W.~L.~Ditto}, {\em Controlling chaos in the brain}, Nature, 370
  (1994), pp.~615--620.

\bibitem{schultz1998}
{\sc W~.Schultz}, {\em Predictive reward signal of dopamine neurons},
  Journal of Neurophysiology, 80 (1998), pp.~1--27.

\bibitem{silverman}
{\sc S.~Silverman}, {\em On maps with dense orbits and the definition of
  chaos}, Rocky Mountain J. Math., 22 (1992), pp.~353--375.

\bibitem{tabak2011}
{\sc J~.Tabak, M~.Tomaiuolo, A~.E~.Gonzalez-Iglesias, L~.S~.Milescu, and R~.Bertram}, {\em Fast-activating voltage-and
  calcium-dependent potassium (bk) conductance promotes bursting in pituitary
  cells: a dynamic clamp study}, The Journal of Neuroscience, 31 (2011),
  pp.~16855--16863.

\bibitem{thunberg}
{\sc H.~Thunberg}, {\em Periodicity versus chaos in one-dimensional dynamics},
  SIAM Review, 43 (2001), pp.~3--30.

\bibitem{touboul:08}
{\sc J.~Touboul}, {\em Bifurcation analysis of a general class of nonlinear
  integrate-and-fire neurons}, SIAM Journal on Applied Mathematics, 68 (2008),
  pp.~1045--1079.

\bibitem{touboul:08b}
\leavevmode\vrule height 2pt depth -1.6pt width 23pt, {\em Nonlinear and
  stochastic models in neuroscience}, PhD thesis, Ecole Polytechnique, Dec.
  2008.

\bibitem{touboul:08c}
\leavevmode\vrule height 2pt depth -1.6pt width 23pt, {\em Sensitivity to the cutoff value in the quadratic
  adaptive integrate-and-fire model}, Research Report 6634, INRIA, Aug. 2008.

\bibitem{touboul:09}
\leavevmode\vrule height 2pt depth -1.6pt width 23pt,  {\em Importance of the cutoff value in the quadratic adaptive
  integrate-and-fire model}, Neural Comput., 21 (2009), pp.~2114--2122.

\bibitem{touboul-brette:08}
{\sc J.~Touboul and R.~Brette}, {\em Dynamics and bifurcations of the adaptive
  exponential integrate-and-fire model}, Biological Cybernetics, 99 (2008),
  pp.~319--334.

\bibitem{touboul-brette:09}
\leavevmode\vrule height 2pt depth -1.6pt width 23pt, {\em Spiking dynamics of
  bidimensional integrate-and-fire neurons}, SIAM Journal on Applied Dynamical
  Systems, 8 (2009), pp.~1462--1506.

\bibitem{Strien}
{\sc S.~van Strien}, {\em Hyperbolicity and invariant measures for general c2
  interval maps satisfying the misiurewicz condition}, Comm. Math. Phys., 128
  (1990), p.~437?495.

\end{thebibliography}
\end{document}